\documentclass[a4paper,11pt,reqno]{amsart} 
\usepackage{amssymb, amsmath, amsthm, mathtools, mathrsfs, bm}
\usepackage{extarrows}
\usepackage[all]{xy}
\usepackage{tikz-cd}
\usepackage{extarrows}
\usepackage{arydshln} 
\usepackage{enumerate}
\usepackage{booktabs,float}
\usepackage{bbm}
\usepackage{xfrac}
\usepackage[square, comma, sort, numbers]{natbib}
\usepackage[colorlinks = true,
            linkcolor = blue,
            urlcolor  = cyan,
            citecolor = black,
            anchorcolor = black]{hyperref}

\numberwithin{equation}{section}

\theoremstyle{plain}
\newtheorem{theorem}{Theorem}[section]

\newtheorem{corollary}[theorem]{Corollary}
\newtheorem{lemma}[theorem]{Lemma}

\theoremstyle{definition}

\theoremstyle{remark}
\newtheorem{remark}[theorem]{Remark}

\newcommand{\xra}{\xrightarrow}
\newcommand{\CP}[1]{\mathbb{C}P^{#1}}

\newcommand{\an}[1]{\ensuremath{\mathbf{A}_{n}^{#1}}}
\newcommand{\PP}{\ensuremath{\mathcal{P}^1}}

\newcommand{\Z}{\ensuremath{\mathbb{Z}}}

\newcommand{\z}[1]{\ensuremath{\mathbb{Z}/2^{#1}}}
\newcommand{\zp}[1]{\ensuremath{\mathbb{Z}/p^{#1}}}

\newcommand{\lara}[1]{\ensuremath{\langle #1 \rangle}}
\newcommand{\ov}[1]{\ensuremath{\overline{ #1 }}}

\DeclareMathOperator{\coker}{coker}
\DeclareMathOperator{\im}{im}
\DeclareMathOperator{\Sq}{Sq}

\newcommand{\mat}[4]{\ensuremath{\ensuremath{\left[\begin{smallmatrix}
  #1&#2\\
  #3&#4
\end{smallmatrix}\right]}}}

\newcommand{\matwo}[2]{\ensuremath{\footnotesize{\begin{bmatrix}
  #1\\
  #2
\end{bmatrix}}}}

\newcommand{\smatwo}[2]{\ensuremath{\left[\begin{smallmatrix}
 #1\\
 #2
\end{smallmatrix}\right]}}

\title{The Homotopy Types of Suspended Simply-connected $6$-manifolds}

\author[P. Li]{Pengcheng Li}
\address{Department of Mathematics, School of Sciences, Great Bay University, Dongguan \rm{523000}, China}
\email{lipcaty@outlook.com}
\urladdr{https://lipacty.github.io}

\author[Z. Zhu]{Zhongjian Zhu}
\address{College of Mathematics and Physics, Wenzhou University, Wenzhou \rm{325035}, China}
\email{zhuzhongjian@amss.ac.cn}



\begin{document}

\begin{abstract}
Under the assumption that certain Adem cohomology operation acts trivially on $H^2(M;\z{})$, we  determine the homotopy types of the triple suspension $\Sigma^3M$ of a simply-connected oriented closed topological(or smooth) $6$-manifold $M$, whose integral homology groups can have $2$-torsion. 
\end{abstract}
\maketitle


\section{Introduction}\label{sect:intro}


In \cite{ST19} So and Theriault found that the homotopy decomposition of the (multiple) suspension space of oriented closed manifolds have direct or implicit applications to the characterization of important concepts in geometry and physics, such as topological $K$-theory, gauge groups and current groups.
Since them, the topic of suspension homotopy of manifolds becomes popular, for instance \cite{Huang21,Huang22, Huang-6mflds, CS22,HL,lipc23}. As can be seen in these literatures, the $2$-torsion of the homology groups of the given manifolds usually causes great obstruction to obtain a complete characterization of the homotopy decomposition of the (multiple) suspension. 


When $M$ is a simply-connected oriented closed (topological or smooth) $6$-manifold, Huang \cite{Huang-6mflds} firstly determined the homotopy type of the double suspension $\Sigma^2 M$ under the assumption that $T$ contains no $2$- or $3$-torsions, while Cutler and So \cite{CS22} obtained the homotopy decompositions of $\Sigma M$ after localization away from $2$.
The purpose of this paper is to study the homotopy types of the suspended spaces of simply-connected $6$-manifolds localized at $2$ and further to obtain a complete homotopy classification of the suspended $6$-manifolds under certain assumptions.

Compared to previous works on this topic, the innovations of the techniques used in our work are reflected in the following three aspects:
\begin{enumerate}
	\item  Instead of the basic method of the homology decomposition used in other works, here we use the classification of homotopy types of $(n-1)$-connected finite CW-complexes of dimension at most $n+2$ \cite{Chang50} to determine the homotopy type of the codimension $1$ cellular skeleton of $\Sigma M$;
	\item Much homotopy theory of a class of the indecomposable CW-complexes, called \emph{elementary Chang-complexes} \cite{Chang50}, is used;
	\item The method given in \cite{PanZhuAn4,ZhuPanAn5} is firstly used to get the homotopy wedge decomposition of the (multiple) suspension of the manifold $M$, which is the mapping cone of some map $f\colon S^{2n+1}\to\bigvee_{i=1}^mX_i$, we find appropriate self-homotopy equivalences of $\bigvee_{i=1}^mX_i$  to reduce  $f_{(2)}$  without affecting $f_{(3)}$, and vice versa, where $f_{(p)}$ is the $p$-localization of $f$ for $p=2,3$. In other words, the reduction of  $f_{(2)}$ and $f_{(3)}$ can be independently dealt with and then combine them to get the reduction of $f$.
\end{enumerate}

To make sense of the introduction and state our main results, we need the following notations and conventions. For a finitely generated group $G$ and an integer $n\geq 2$, denote by $P^n(G)$ the $n$-dimensional Peterson space which is characterized by its unique non-trivial reduced integral cohomology group $G$ in dimension $n$. For $n,k\geq 2$, denote $P^n(\Z/k)$ by $P^n(k)$.
The Moore space of homotopy type $(\Z/k,n-1)$ and $\Z/k$ is the group of integers mod $k$. There is a canonical homotopy cofibre sequence for the elementary Moore spaces $P^n(p^r)$ with prime number $p$
\begin{align}
S^{n-1}\xra{p^r}S^{n-1}\xra{i^r_{n-1}}P^n(p^r)\xra{q^r_n}S^n \label{Cof Moore}
\end{align}
where $p^r$ is the degree $p^r$ map, $i^r_{n-1}$ and $q^r_n$ are the canonical inclusion and pinch maps, respectively.  A finite CW-complex is called an \emph{$\an{k}$-complex} if it is $(n-1)$-connected and has dimension at most $n+k$. An $\an{k}$-complex is \emph{indecomposable} if it cannot be homotopically decomposed as a wedge sum of two non-contratible complexes. It is a basic fact that indecomposable $\an{1}$-complexes with $n\geq 2$ consist of spheres and elementary Moore spaces.
Due to Chang\cite{Chang50}, for $n\geq 3$, indecomposable homotopy types of  $\an{2}$-complexes constitute spheres $S^{k} (k=n,n+1,n+2)$; elementary Moore spaces $P^{k}(p^r) (k=n+1,n+2)$ and the following four classes of  \emph{elementary Chang-complexes}:
	\begin{align*}
		C^{n+2}_\eta&=S^n\cup_{\eta} e^{n+2}=\Sigma^{n-2} \CP{2},\quad  C^{n+2}_r=P^{n+1}(2^r)\cup_{i^r_{n}\eta}e^{n+2},\\
		C^{n+2,s}&=S^n\cup_{\eta q^s_{n+1}}\bm C P^{n+1}(2^s), \quad C^{n+2,s}_r=P^{n+1}(2^r)\cup_{i^r_n\eta q^s_{n+1}} \bm C P^{n+1}(2^s).
	\end{align*}
where $r,s$ are positive integers and $\eta=\eta_n:S^{n+1}\to  S^n$ is the (iterated) suspension of the Hopf map $\eta\colon S^3\to S^2$. Let $k\geq 0$ be an integer. We denote the wedge sum $\bigvee_{i=1}^kX$ of $k$ copies of a pointed space $X$ by $kX$ for simplicity.  If $V=\bigvee_{i=1}^{k}X_i$,  denote the wedge sum $\bigvee_{i=1,j\neq i_0}^{k}X_i$ by $V/X_{i_0}$. For a prime $p$, denote by $X_{(p)}$ the $p$-localization of the space $X$ and denote by $X\simeq_{(p)}Y$ if there is a homotopy equivalence between $X$ and $Y$ when localized at $p$.

We also need the following stable cohomology operations. For each prime $p$ and integer $r\geq 1$, there are the higher order Bocksteins $\beta_r$ which detects the degree $2^r$ map on spheres $S^n$. 
For each $n\geq 2$, there are secondary cohomology operations
\begin{equation}\label{eq:Theta}
  \Theta_n\colon S_n(X)\to T_n(X)
 \end{equation}  
 detecting the map $\eta^2\in\pi_{n+2}(S^n)$, where 
 \begin{align*}
  S_n(X)&=\ker(\theta_n)_\sharp=\ker(\Sq^2)\cap \ker(\Sq^2\Sq^1)\\
  T_n(X)&=\coker(\Omega\varphi_n)_\sharp=H^{n+3}(X;\z{})/\im(\Sq^1+\Sq^2). 
\end{align*} 
That is, $\Theta_n$ is based on the null-homotopy of the composition
\[K_n\xra{\theta_n=\smatwo{\Sq^2\Sq^1}{\Sq^2}} K_{n+3}\times K_{n+2}\xra{\varphi_n=[\Sq^1,\Sq^2]}K_{n+4},\]
where $K_m=K(\z{},m)$ denotes the Eilenberg-MacLane space of type $(\z{},m)$.
The map $\eta^3$ can be detected by the tertiary operation $\mathbb{T}$, see \cite[Exercise 4.2.5]{Harperbook}. 
Recall that there is a unique stable secondary operation $\Psi$ of degree $4$, also called the \emph{Adem operation}, based on the Adem relation 
 \[\Sq^4\Sq^1+\Sq^2\Sq^3+\Sq^1\Sq^4=0.\]  
Note that$\Psi$ is defined on classes which annihilated by each $\Sq^1,\Sq^3,\Sq^4$; 
  $\Psi$ detects $2\nu_n$ \cite[Lemma 4.4.4]{Adams60}, where $\nu_n\in \pi_{n+3}(S^n)$ is the Hopf map. For $m\geq 3$, let $\alpha_1(m)=\Sigma^{m-3}\alpha_1(3)$
  be the generator of the $3$-primary component of $\pi_{m+3}(S^m)$. Denote by $\tilde{\alpha}_1(m)$ the possible  lift (if exists) of $\alpha_1(m)$ to $P^m(3^r)$; i.e., $q_m^r\circ \tilde{\alpha}_1(m)=\alpha_1(m)$.
  See \cite[Lemma 2.6]{LPW} or \cite[Chapter XIII]{Todabook} with $p=3$. Note that $\alpha_1(m)$ and $\tilde{\alpha}_1(m)$ can be detected by the Steenrod reduced third power operation $\PP$  (cf. \cite[Chapter 1.5.5]{Harperbook}).
 
Let $M$ be a simply-connected closed oriented topological $6$-manifold whose integral homology groups $H_\ast(M)$ are given by the following table
\begin{equation}\label{HM}
  \begin{tabular}{cccccccc}
    \toprule
$i$& $2$&$3$&$4$&$0,6$&$\text{otherwise}$
  \\\midrule
  $H_i(M)$& $\Z^l\oplus T$&  $\Z^d\oplus T$ & $\Z^l$& $\Z$&$0$
  \\ \bottomrule
  \end{tabular},
\end{equation}
where $T$ is a finitely generated torsion group. 
Let $T_2$ and $T_3$ be the $2$ and $3$-primary components of the torsion group $T$, respectively. Denote  
  \[T= T_2\oplus T_{\geq 3}=T_2\oplus T_3\oplus  T_{\geq 5},\quad  T_2=\bigoplus_{i=1}^{m_2}\z{r_i},~~ T_3=\bigoplus_{j=1}^{m_3}\Z/{3^{r'_j}}.\] 
Under the assumption that the Adem operation $\Psi$ acts trivially on $H^2(M;\z{})$, we determine the $2$-local and integral homotopy types of $\Sigma^3 M$ by the following Theorem \ref{thm:6mfds 2-local} and Theorem \ref{thm:6mfds total}, respectively, in which all the homotopy classes of the attaching maps of the top cell $e^9$ are defined in Section \ref{sec:An2}, especially in Lemma \ref{lem:Moore1}-Lemma \ref{lemma pi(Changcplex)}. 

\begin{theorem}\label{thm:6mfds 2-local}
  Let $M$ be a simply-connected closed oriented $6$-manifold with $H_\ast(M)$ given by (\ref{HM}).
   Assume that the Adem operation $\Psi$ acts trivially on $H^2(M;\z{})$. There are non-negative integers $k, t_i (i=0,1,2,3,4)$, $r_j (0\leq j\leq t_1)$, $\bar{r}_j (0\leq j\leq t_2)$,  $\check{r}_j (0\leq j\leq t_4)$,  $s_j (0\leq j\leq t_0)$, $\hat{s}_j (0\leq j\leq t_3)$,  $\check{s}_j (0\leq j\leq t_4)$ such that $0\leq k+t_2, k+t_3\leq l$ and the equations (\ref{equ:ri}) and (\ref{equ:si}) hold. Denote 
   \begin{equation}\label{equ:V7}
	\begin{aligned}
		V_7=dS^6\vee (l-k-t_3)S^{5}\vee k C_{\eta}^{7}\vee (l-k-t_2) S^{7}\vee (\bigvee_{j=1}^{t_0}P^{7}(2^{s_j}))\\
	\vee(\bigvee_{j=1}^{t_1}P^{6}(2^{r_j}))  \vee (\bigvee_{j=1}^{t_2}C_{\bar{r}_{j}}^{7})\vee (\bigvee_{j=1}^{t_3}C^{7,\hat{s}_j})\vee (\bigvee_{j=1}^{t_4}C_{\check{r}_{j}}^{7,\check{s}_j}). 
	\end{aligned}
	\end{equation}
  \begin{enumerate}[1.]
   \item\label{thm-spin Theta=0, T=0} 
  	 Suppose that the Steenrod square $\Sq^2$ acts trivially on $H^4(M;\z{})$.
	 \begin{enumerate}
		\item If $\Theta$ and  $\mathbb{T}$ act trivially on $H^3(M;\z{})$  and  $H^2(M;\z{})$ respectively, then there is a homotopy equivalence  
		\[\Sigma^3 M\simeq_{(2)} S^9\vee V_7.\] 
	
   \item\label{thm-spin Theta neq0} If $\Theta$ acts non-trivially on $H^3(M;\z{})$,  then the $2$-local homotopy type of $\Sigma^3 M$ is one of the following 
\begin{align*}
	\big(V_7/P^{7}(2^{s_{j_0}})\big)\vee \big(P^{7}(2^{s_{j_0}})\cup_{i_6^{s_{j_0}}\eta^2}e^9\big),&\quad \big(V_7/P^{6}(2^{r_{j_0}})\big)\vee \big(P^{6}(2^{r_{j_0}})\cup_{\tilde{\eta}_{r_{j_0}}\eta}e^9\big),\\[1ex]
	\big(V_7/C^{7}_{\bar{r}_{j_0}}\big)\vee \big(C^{7}_{\bar{r}_{j_0}}\cup_{\bar{i}_P^{\bar{r}_{j_0}}\tilde{\eta}_{\bar{r}_{j_0}}\eta}e^9\big), &\quad \big(V_7/C^{7,\hat{s}_{j_0}}\big)\vee \big(C^{7,\hat{s}_{j_0}}\cup_{\hat{i}^{\hat{s}_{j_0}}_6\eta^2}e^9\big), \\[1ex]
	\big(V_7/C^{7,\check{s}_{j_0}}_{\check{r}_{j_0}}\big)\vee \big(C^{7,\check{s}_{j_0}}_{\check{r}_{j_0}}\cup_{\check{i}_P^{\check{r}_{j_0}}\tilde{\eta}_{\check{r}_{j_0}}\eta}e^9\big),&\quad \big(V_7/C^{7,\check{s}_{j_0}}_{\check{r}_{j_0}}\big)\vee \big(C^{7,\check{s}_{j_0}}_{\check{r}_{j_0}}\cup_{\check{i}^{\check{s}_{j_0},\check{r}_{j_0}}_6\eta^2}e^9\big).
\end{align*}

    \item\label{thm-spin Theta=0, Tneq0} If $\Theta$  acts trivially on $H^3(M;\z{})$, while $\mathbb{T}$ acts non-trivially on $H^2(M;\z{})$, then the $2$-local homotopy type of $\Sigma^3 M$ is one of the following 
\[ \big(V_7/S^5\big)\vee \big(S^5\cup_{\eta^3}e^9\big), \quad \big(V_7/P^{6}(2^{r_{j_0}})\big)\vee \big(P^{6}(2^{r_{j_0}})\cup_{i_5^{r_{j_0}}\eta^3}e^9\big) (r_{j_0}\geq 3).\] 
	 \end{enumerate}

  \item\label{thm-notspin}  Suppose that $\Sq^2$ acts non-trivially on $H^4(M;\z{})$, then there is a homotopy equivalence 
  \[\Sigma^3 M\simeq_{(2)} \big(V_7/P^{7}(2^{s_{j_0}})\big)\vee \big(P^{7}(2^{s_{j_0}})\cup_{\tilde{\eta}_{s_{j_0}}}e^9\big), ~~\text{ or }~~ \Sigma^3 M\simeq_{(2)} \big(V_7/S^7\big)\vee C_{\eta}^9.\] 
    
   \end{enumerate}
\end{theorem}

Motivated by the case (\ref{3-local:3}) of Theorem \ref{thm:3-local}, we define 
\begin{flushleft}
	\textbf{Condition $\star$}: There exist some cohomology classes $u,v\in H^{3}(\Sigma M;\Z/3)$ and some index $j_0'$ such that 
\[\PP(u)\neq 0, ~\PP(v)=0, \text{ and }\beta_{r'_{j'_0}}(u+v)\neq 0.\]
\end{flushleft}

\begin{theorem}\label{thm:6mfds total}
	Let $M$ be a simply-connected closed oriented $6$-manifold satisfying the assumption in Theorem \ref{thm:6mfds 2-local}.
	There are non-negative integers $k, t_i (i=0,1,2,3,4)$, $r_j (0\leq j\leq t_1)$, $\bar{r}_j (0\leq j\leq t_2)$,  $\check{r}_j (0\leq j\leq t_4)$,  $s_j (0\leq j\leq t_0)$, $\hat{s}_j (0\leq j\leq t_3)$,  $\check{s}_j (0\leq j\leq t_4)$ such that $0\leq k+t_2, k+t_3\leq l$ and the equations (\ref{equ:ri}) and (\ref{equ:si}) hold. Denote 
	\begin{equation}\label{equ:W7}
		\begin{aligned}
			W_7=dS^6\vee P^6(T)\vee P^7(T)\vee  (l-k-t_3)S^{5}\vee k C_{\eta}^{7}\vee (l-k-t_2) S^{7}\\
			 \vee (\bigvee_{j=1}^{t_2}C_{\bar{r}_{j}}^{7})\vee (\bigvee_{j=1}^{t_3}C^{7,\hat{s}_j})\vee (\bigvee_{j=1}^{t_4}C_{\check{r}_{j}}^{7,\check{s}_j}). 
		\end{aligned}
		\end{equation}
	\begin{enumerate}[1.]
	\item \label{Thm:P1=0}
	Suppose that  $\PP\colon H^{3}(\Sigma M;\Z/3)\to H^{7}(\Sigma M;\Z/3)$ is trivial. 
	\begin{enumerate}
		\item If the operations $\Sq^2$,  $\Theta$ and  $\mathbb{T}$ act trivially on $H^4(M;\z{})$, $H^3(M;\z{})$  and  $H^2(M;\z{})$ respectively,  then there is a homotopy equivalence 
		\[\Sigma^3 M \simeq S^9\vee W_7.\]

		\item\label{thm-total:spin Theta neq0} If  $\Sq^2$ acts trivially on $H^4(M;\z{})$, while $\Theta$ acts non-trivially on $H^3(M;\z{})$,  then the homotopy type of $\Sigma^3 M$ is one of the following 
	\begin{align*}
		\big(W_7/P^{7}(2^{s_{j_0}})\big)\vee \big(P^{7}(2^{s_{j_0}})\cup_{i_6^{s_{j_0}}\eta^2}e^9\big), &\quad \big(W_7/P^{6}(2^{r_{j_0}})\big)\vee \big(P^{6}(2^{r_{j_0}})\cup_{\tilde{\eta}_{r_{j_0}}\eta}e^9\big),\\[1ex]
		\big(W_7/C^{7}_{\bar{r}_{j_0}}\big)\vee \big(C^{7}_{\bar{r}_{j_0}}\cup_{\bar{i}_P^{\bar{r}_{j_0}}\tilde{\eta}_{\bar{r}_{j_0}}\eta}e^9\big),  &\quad \big(W_7/C^{7,\hat{s}_{j_0}}\big)\vee \big(C^{7,\hat{s}_{j_0}}\cup_{\hat{i}^{\hat{s}_{j_0}}_6\eta^2}e^9\big),  \\[1ex]
		\big(W_7/C^{7,\check{s}_{j_0}}_{\check{r}_{j_0}}\big)\vee \big(C^{7,\check{s}_{j_0}}_{\check{r}_{j_0}}\cup_{\check{i}_P^{\check{r}_{j_0}}\tilde{\eta}_{\check{r}_{j_0}}\eta}e^9\big),&  \quad \big(W_7/C^{7,\check{s}_{j_0}}_{\check{r}_{j_0}}\big)\vee \big(C^{7,\check{s}_{j_0}}_{\check{r}_{j_0}}\cup_{\check{i}^{\check{s}_{j_0},\check{r}_{j_0}}_6\eta^2}e^9\big).
	\end{align*}
		\item If $\Sq^2$ and $\Theta$ act trivially on $H^4(M;\Z{})$ and $H^3(M;\z{})$, respectively, while $\mathbb{T}$ acts non-trivially on $H^2(M;\z{})$, then the homotopy type of $\Sigma^3 M$ is one of the following 
	\[\big(W_7/S^5\big)\vee \big(S^5\cup_{\eta^3}e^9\big), \quad \big(W_7/P^{6}(2^{r_{j_0}})\big)\vee \big(P^{6}(2^{r_{j_0}})\cup_{i_5^{r_{j_0}}\eta^3}e^9\big) (r_{j_0}\geq 3).\]
		
		\item If $\Sq^2$ acts non-trivially on $H^4(M;\z{})$.  then there is a homotopy equivalence
		\[\Sigma^3 M\simeq \big(W_7/P^{7}(2^{s_{j_0}})\big)\vee \big(P^{7}(2^{s_{j_0}})\cup_{\tilde{\eta}_{s_{j_0}}}e^9\big), \text{ or }~  \Sigma^3 M\simeq \big(W_7/S^7\big)\vee C_{\eta}^9.\]
		
	\end{enumerate}
	
		\item \label{Thm:P1 neq 0}
	Suppose that  $\PP\colon H^{3}(\Sigma M;\Z/3)\to H^{7}(\Sigma M;\Z/3)$ is non-trivial. 
	Define the space $X$ by 
	\[
		X=\left\{\begin{array}{ll}
			P^6(3^{r'_{j_0'}}),& \text{ if \textbf{Condition $\star$} holds};\\[1ex]
			S^5,& \text{ if \textbf{Condition $\star$} fails and $l-k-t_3\neq 0$};\\[1ex]
			C^7_\eta,& \text{ if \textbf{Condition $\star$} fails, $l-k-t_3= 0$ and $k\neq 0$};\\[1ex]
			C^{7,\hat{s}_1},& \text{ if \textbf{Condition $\star$} fails, $l-k-t_3= 0$ and $k\neq 0$}.
		\end{array}\right. 
	  \]
	Denote by $\alpha_X=\alpha_{1}(5), i_5^{\eta}\alpha_{1}(5), \hat{i}^{\hat{s}_{1}}_5\alpha_{1}(5)$ and  $i_{5}^{r'_{j}}\alpha_{1}(5)$ for $X=S^5, C_{\eta}^7, C^{7,\hat{s}_{1}}$ and $P^{6}(3^{r'_{j_0}})$, respectively.
	\begin{enumerate}
		\item\label{Thm:P1 neq 0:a} If $\Sq^2$,  $\Theta$ and  $\mathbb{T}$ act trivially on $H^3(M;\z{})$  and  $H^2(M;\z{})$ respectively,  then there is a homotopy equivalence
	    \[\Sigma^3 M \simeq  \big(W_7/X\big)\vee  \big(X\cup_{\alpha_X}e^9\big).\]

		\item\label{Thm:P1 neq 0:b} If $\Sq^2$ acts trivially on $H^4(M;\z{})$, while $\Theta$ acts non-trivially on $H^3(M;\z{})$,  then the homotopy type of $\Sigma^3 M$ is one of the following
		\begin{enumerate}
			\item\label{Thm:P1 neq 0:b:1} $\big(W_7/(P^{7}(2^{s_{j_0}})\vee X)\big)\vee \big((P^{7}(2^{s_{j_0}})\vee X)\cup_{\smatwo{i_6^{s_{j_0}}\eta^2}{\alpha_X}}e^9\big)$;
			\item\label{Thm:P1 neq 0:b:2}   $\big(W_7/(P^{6}(2^{r_{j_0}})\vee X)\big)\vee \big((P^{6}(2^{r_{j_0}})\vee X)\cup_{\smatwo{\tilde{\eta}_{r_{j_0}}\eta}{\alpha_X}}e^9\big)$;
			\item\label{Thm:P1 neq 0:b:3}  $\big(W_7/(C^{7}_{\bar{r}_{j_0}} \vee X)\big)\vee \big((C^{7}_{\bar{r}_{j_0}} \vee X)\cup_{\smatwo{\bar{i}_P^{\bar{r}_{j_0}}\tilde{\eta}_{\bar{r}_{j_0}}\eta}{\alpha_X}}e^9\big)$; 
			\item\label{Thm:P1 neq 0:b:4}   $\big(W_7/(C^{7,\hat{s}_{j_0}}\vee X)\big)\vee \big((C^{7,\hat{s}_{j_0}}\vee X)\cup_{\smatwo{\hat{i}^{\hat{s}_{j_0}}_6\eta^2}{\alpha_X}}e^9\big)$; 
			\item\label{Thm:P1 neq 0:b:5}  $\big(W_7/C^{7,\hat{s}_{j_0}}\big)\vee \big(C^{7,\hat{s}_{j_0}}\cup_{\hat{i}^{\hat{s}_{j_0}}_6\eta^2+ \hat{i}^{\hat{s}_{j_0}}_5\alpha_{1}(5)}e^9\big)$; 
			\item\label{Thm:P1 neq 0:b:6} $\big(W_7/(C^{7,\check{s}_{j_0}}_{\check{r}_{j_0}}\vee X)\big)\vee \big((C^{7,\check{s}_{j_0}}_{\check{r}_{j_0}}\vee X)\cup_{\smatwo{\check{i}_P^{\check{r}_{j_0}}\tilde{\eta}_{\check{r}_{j_0}}\eta}{\alpha_X}}e^9\big)$;
			\item\label{Thm:P1 neq 0:b:7} $\big(W_7/(C^{7,\check{s}_{j_0}}_{\check{r}_{j_0}}\vee X)\big)\vee \big((C^{7,\check{s}_{j_0}}_{\check{r}_{j_0}}\vee X)\cup_{\smatwo{\check{i}^{\check{s}_{j_0},\check{r}_{j_0}}_6\eta^2}{\alpha_X}}e^9\big)$. 
		\end{enumerate}
		
		\item\label{Thm:P1 neq 0:c}  If  $\Sq^2$ and $\Theta$  act trivially on $H^4(M;\z{})$ and $H^3(M;\z{})$, respectively, while $\mathbb{T}$ acts non-trivially on $H^2(M;\z{})$, then the homotopy type of $\Sigma^3 M$ is one of the following 
		\begin{enumerate}
			\item\label{Thm:P1 neq 0:c:1} $ \big(W_7/S^5\big)\vee \big(S^5\cup_{4\nu}e^9\big)$;
			\item\label{Thm:P1 neq 0:c:2} $ \big(W_7/(S^5\vee P^{6}(3^{r'_{j'_0}}))\big)\vee \big((S^5\vee P^{6}(3^{r'_{j'_0}}))\cup_{\smatwo{\eta^3}{i_{5}^{r'_{j'_0}}\alpha_{1}(5)}}\!\!e^9\big)$;
			\item\label{Thm:P1 neq 0:c:3} 
			$\big(W_7/(P^{6}(2^{r_{j_0}})\vee X)\big)\vee \big((P^{6}(2^{r_{j_0}})\vee X)\cup_{\smatwo{i_5^{r_{j_0}}\eta^3}{\alpha_X}}e^9\big)$, $r_{j_0}\geq 3$. 
		\end{enumerate}
		
		\item\label{Thm:P1 neq 0:nonspin} If  $\Sq^2$ acts non-trivially on $H^4(M;\z{})$, then 
	the homotopy type of $\Sigma^3 M$ is one of the following
	\begin{enumerate}
		\item $\big(W_7/(P^{7}(2^{s_{j_0}})\vee X)\big)\vee \big((P^{7}(2^{s_{j_0}})\vee X)\cup_{\smatwo{\tilde{\eta}_{s_{j_0}}}{\alpha_X}}e^9\big);$
		\item $\big(W_7/(S^7\vee X)\big)\vee \big((S^7\vee X)\cup_{\smatwo{\eta}{\alpha_X}}e^9\big).$
	\end{enumerate}

	\end{enumerate}

\end{enumerate}
	
\end{theorem}

If, in addition, the $6$-manifold $M$ in the above two theorems is smooth,  then the case  (\ref{thm-spin Theta=0, Tneq0}) in Theorem \ref{thm:6mfds 2-local} and the cases (\ref{thm-total:spin Theta neq0}),  (\ref{Thm:P1 neq 0:b}) in Theorem \ref{thm:6mfds total} can be removed, since the secondary operation $\Theta$  acts trivially on the mod $2$ cohomology ring of spin manifolds, see \cite{Thomas67} or \cite[page 32-33]{MMbook}.  

Recall that there hold the Wu formulae \cite{Wu54,MS74}:
\begin{align*}
	&\Sq^2(x)=w_2(M)\smallsmile x \text{ for any $x\in H^4(M;\z{})$},\\
	&\PP(x)=(p_1(M)\text{ mod } 3)\smallsmile x\text{ for any $x\in H^2(M;\Z/3)$},
\end{align*}
	where $w_2(M)\in H^2(M;\z{})$ and $p_1(M)\in H^4(M;\Z)$ are the second Stiefel-Whitney class and the first Pontryagin class of the tangent bundle of $M$, respectively.
It follows that one may alternatively replace the conditions that the Steenrod square $\Sq^2$ acts trivially on $H^4(M;\z{})$ by $w_2(M)=0$, and that $\PP$ acts trivially on $H^3(\Sigma M;\z{})$ by the mod $3$ reduction of $p_1(M)$ and a homotopy invariant $\mathrm{rad}_{p_1}(M)$, called the $p_1$-radius of $M$; see \cite{HL} for more details.  

On the other hand, although we list all the possibilities of homotopy types of $\Sigma^3 M$ of the given $6$-manifolds, we cannot guarantee that all possibilities could happen or find a concrete $6$-manifold as an illustrating example. This is partially due to our lack of knowledge about manifolds. 

The paper is arranged as follows. In Section \ref{sec:An2} we introduce the global notations and computed some homotopy groups of $\an{2}$-complexes, $n\geq 3$. Section \ref{sec:coho-operat} covers some auxiliary lemmas on cohomology operations mentioned in Section \ref{sect:intro}. In Section \ref{sect:proofs} we analyze the homotopy type of the triple suspension $\Sigma^3M$ and prove Theorem \ref{thm:6mfds 2-local} and \ref{thm:6mfds total}.

\subsection*{Acknowledgements}
Pengcheng Li was supported by National Natural Science Foundation of China (Grant No. 12101290), Zhongjian Zhu was supported by National Natural Science Foundation of China
(Grant No. 11701430).

\section{Notations and some homotopy theory of $\an{2}$-complexes}\label{sec:An2}

Throughout the paper all spaces are based CW-complexes, all maps are base-point-preserving and are identified with their homotopy classes in notation. We shall globally use the following matrix notations to denote maps between wedge sums of spaces.
Let $X=\Sigma X'$, $Y_i=\Sigma Y_i'$ be suspensions, $i=1,2,\cdots,n$. Let \[j_l\colon Y_l\to \bigvee_{i=1}^nY_i,\quad p_k\colon \bigvee_{i=1}^n Y_i\to Y_k\] be respectively the canonical inclusions and projections, $1\leq k,l\leq n$. By the Hilton-Milnor theorem,  we can denote a map
\(f\colon X\to\bigvee_{i=1}^nY_i\)
by \begin{align}
	f=\sum_{k=1}^n j_k\circ f_{k}+\theta_f (\text{or }f=\sum_{k=1}^n  f_{k}+\theta_f) , \label{equ:f}
\end{align}
where $f_{k}=p_k\circ f\colon X\to Y_k$, $\theta$ is a linear combination of Whitehead products and satisfies $\Sigma \theta=0$.
The part $\sum_{k=1}^n j_k\circ f_{k}$ of (\ref{equ:f}) is conveniently denoted by a column vector $u_f=[f_1,f_2,\cdots,f_n]^t$, so $f=u_f+\theta_f$. 
 We denote a self-map $h$ of $\bigvee_{i=1}^nY_i$ with $\theta_h=0$ in the matrix form 
\begin{align}
M_h\coloneqq [h_{kl}]_{n\times n}=\begin{bmatrix}
	h_{11}&h_{12}&\cdots&h_{1n}\\
	h_{21}&h_{22}&\cdots&h_{2n}\\
	\vdots&\vdots&\ddots&\vdots\\
	h_{n1}&h_{n1}&\cdots&h_{nn}
\end{bmatrix}, \label{matrix (h)}
\end{align}
where $h_{kl}=p_k\circ h\circ j_l\colon Y_l\to Y_k$.
If $X_j$, $Y_i$ are suspended spaces and $h\colon \bigvee_{j=1}^{m}X_j\to  \bigvee_{i=1}^{n} Y_i$ is a map determined by its components $h_{kl}=p_khj_l\colon X_l\to  Y_k$, then we also denote $h$ in the matrix form $h=[h_{kl}]_{n\times m}$.
Given another map $g\colon X\to\bigvee_{i=1}^nY_i$,
the composition law
\(h(f+g)=h f+h g\)
implies that the matrix multiplication
\(M_h[f_1,f_2,\cdots,f_n]^t\)
 represents the composite $h\circ f$.
We call two maps $f=u_f$ and $g=u_g$ are \emph{equivalent}, denoted by $f\sim g$ or 
\[[f_1,f_2,\cdots,f_n]^t\sim [g_1,g_2,\cdots,g_n]^t,\]
if there is a self-homotopy equivalence $h$ of $\bigvee_{i=1}^n Y_i$, which can be represented by the matrix $M_h$, such that 
\[M_h[f_1,f_2,\cdots,f_n]^t=[g_1,g_2,\cdots,g_n]^t.\] 
Note that the above matrix multiplication refers to elementary row operations in matrix theory; and the homotopy cofibres of the maps $f=u_f$ and $g=u_g$ are homotopy equivalent if $f$ and $g$ are equivalent.

The homotopy theory of elementary Chang complexes have been extensively studied, see \cite{ZP17,ZP21,ZLP19,ZP23,lipc22}.  
There are some homotopy cofibre sequences for elementary Chang complexes (denoted by \textbf{Cof. List}, cf. \cite{ZP17,lipc22}), which give some notations for some canonical maps between $\an{2}$-complexes, $n\geq 4$. In particular, the maps endowed with the labels ``$i$'' and ``$q$'' are the canonical  inclusions and pinch maps, respectively.
  
\textbf{Cof. List:}

\begin{enumerate}[1.]
	\item Homotopy cofibre sequence for $C^{n+2}_\eta$: 
	   \[S^{n+1}\xra{\eta_n}S^n\xra{i^{\eta}_n}C^{n+2}_\eta\xra{q^{\eta}_{n+2}}S^{n+2}.\]
	\item Homotopy cofibre sequences for $C^{n+2}_r$: 
	\begin{align*}
	\textbf{Cof1}:~~& S^{n+1}\vee S^{n}\xra{[\eta_n,2^r]}S^{n}\xra{\bar{i}_{n}^r}C^{n+2}_r\xra{\smatwo{\bar{q}_{n+2}^r}{\bar{q}_{n+1}^r}}S^{n+2}\vee S^{n+1},\\
	\textbf{Cof2}:~~	&S^{n+1}\xra{i^{r}_n\eta_n}P^{n+1}(2^r)\xra{\bar{i}_{P}^r}C^{n+2}_r\xra{\bar{q}_{n+2}^r}S^{n+2},\\
	\textbf{Cof3}:~~	&S^{n}\xra{2^ri_{n}^{\eta}}C^{n+2}_\eta\xra{\bar{i}_\eta}C^{n+2}_r\xra{\bar{q}_{n+1}^r}S^{n+1};
	\end{align*}
	\item Homotopy cofibre sequences for $C^{n+2,s}$:  
	 \begin{align*}
		\textbf{Cof1}:~~& S^{n+1}\xra{\smatwo{2^s}{\eta_n}} S^{n+1}\vee S^n\xra{[\hat{i}_{n+1}^{s},\hat{i}_{n}^{s}]}C^{n+2,s}\xra{\hat{q}_{n+2}^{s}}S^{n+2},\\
		\textbf{Cof2}:~~& C_{\eta}^{n+1}\xra{2^sq^{\eta}_{n+1}} S^{n+1}\xra{\hat{i}_{n+1}^{s}}C^{n+2,s}\xra{\hat{q}_{\eta}}C_{\eta}^{n+2},\\
		\textbf{Cof3}:~~& P^{n+1}(2^s)\xra{\eta q_{n+1}^s} S^n\xra{\hat{i}_{n}^{s}}C^{n+2,s}\xra{\hat{q}_{P}}P^{n+2}(2^s);
	 \end{align*}
	 \item Homotopy cofibre sequences for $C^{n+2,s}_r$:	 
\begin{align*}
	\textbf{Cof1}:~~&S^{n+1}\vee S^{n}\xra{\mat{2^s}{\eta_n}{0}{2^r}}S^{n+1}\vee S^{n}\xra{[\check{i}_{n+1}^{s,r},\check{i}_n^{s,r}]}C^{n+2,s}_r\xra{\smatwo{\check{q}_{n+2}^{s,r}}{\check{q}_{n+1}^{s,r}}}S^{n+2}\vee S^{n+1},\\
	\textbf{Cof2}:~~&P^{n+1}(2^s)\xra{i_{n}^r\eta q_{n+1}^s}P^{n+1}(2^r)\xra{\check{i}_P^r}C^{n+2,s}_r\xra{\check{q}_{P}^s}P^{n+2}(2^s),\\
	\textbf{Cof3}:~~& S^n\xra{2^r\hat{i}_n^s} C^{n+2,s}\xra{\check{i}_{C}^s} C^{n+2,s}_r\xra{\check{q}_{n+1}^{s,r}} S^{n+1},\\
	\textbf{Cof4}:~~&C^{n+1}_r\xra{2^s\bar{q}_{n+1}^{r}}S^{n+1}\xra{\check{i}_{n+1}^{s,r}}C^{n+2,s}_r\xra{\check{q}_C^r}C^{n+2}_r.
\end{align*}

\end{enumerate}

\medskip
We will frequently use the following generators of homotopy groups of spheres (cf. \cite{Todabook}):
\begin{enumerate}[(1)]
	\item $\pi_{m+1}(S^m)\cong\z{}\langle \eta_m \rangle$, $\pi_{m+2}(S^m)\cong \z{}\langle \eta_m^2 \rangle$; $\pi_{m+3}(S^m;3)\cong\Z/3\langle \alpha_1(m)\rangle$ for $m\geq 3$, where $\pi_k(X;p)$ denotes the $p$-primary component of $\pi_k(X)$;
	\item  $\pi_6(S^3)\cong\Z/12\langle \nu' \rangle$; $\pi_7(S^4)\cong \Z\langle \nu_4\rangle\oplus\Z/12\langle \Sigma \nu'\rangle$,  $\pi_{m+3}(S^m)\cong\Z/{24}\langle \nu_m \rangle\cong\Z/8\lara{3\nu_m}\oplus \Z/3\lara{\alpha_1(m)}$ for $m\geq 5$. There hold formulae 
 \[\eta_3^3=6\nu',\quad \eta_m^3=12\nu_m,\quad 2\nu_m=\Sigma^{m-3}\nu' \text{ for }m\geq 5.\] 
\end{enumerate}
For simplicity, in the remainder of this paper we denote by $\eta$ and $\nu$ the iterated suspensions $\eta_m (m\geq 3)$ and $\nu_m (m\geq 5)$, respectively. Denote $m_{r}^s=\min\{r,s\}$ for positive integers $r,s$.  Set $\delta_1=1$ and $\delta_r=0$ for $r\geq 2$.  

\begin{lemma}[cf. \cite{BH91}]\label{lem:Moore1}
    Let $n\geq 3,r\geq 1$ be integers and let $p$ be a prime.  The followings hold:
   \begin{enumerate}[(1)]
    \itemsep=1ex
    \item $\pi_n(P^{n+1}(p^r))\cong\zp{r}\langle i^r_n\rangle$.
    \item $\pi_{n+1}(P^{n+1}(p^r))\cong \left\{\begin{array}{ll}
  \Z/2\langle i^r_n\eta\rangle,&p=2;\\
  0,&p\geq 3.
    \end{array}\right.$

    \item $\pi_{n+2}(P^{n+1}(p^r))\cong \left\{\begin{array}{ll}
      \Z/4\langle \tilde{\eta}_1\rangle,&p=2,~r=1;\\
      \Z/2 \langle i^r_n\eta^2\rangle\oplus \Z/2\langle \tilde{\eta}_r \rangle,&p=2,~r\geq 2;\\
      0,&p\geq 3. 
    \end{array}\right.$
    \\
  The generator $\tilde{\eta}_r$ satisfies formulae
  \begin{equation}\label{eq:eta}
    q_{n+1}^r\tilde{\eta}_r=\eta,\quad 2\tilde{\eta}_1=i_n^1\eta^2.
  \end{equation}
  \item\label{Moore1-PP} $[P^{n+1}(2^r),P^{n+1}(2^s)]\cong \left\{\begin{array}{ll}
    \Z/4\langle 1_P\rangle,&r=s=1;\\
    \z{m_r^s}  \langle B(\chi^r_s) \rangle\oplus\Z/2 \langle i^r_n\eta q^r_{n+1}\rangle,&\text{$r$ or $s>1$},
  \end{array}\right.$ \\
where $B(\chi^r_s)$ satisfies $\Sigma B(\chi^r_s)=B(\chi^r_s)$ and 
\begin{equation}\label{eq:chi}
  B(\chi^r_s)i_{n}= \left\{\begin{array}{ll}
        i_{n}^s,&r\geq s;\\
        2^{s-r}i_{n}^s,&r\leq s.
      \end{array}\right. ~~q_{n+1}^sB(\chi^r_s)= \left\{\begin{array}{ll}
        2^{r-s}q_{n+1}^r,&r\geq s;\\
        q_{n+1}^r,&r\leq s.
      \end{array}\right.
\end{equation}
\begin{equation}\label{eq:chi-eta}
  B(\chi^r_s)\tilde{\eta}_r= \tilde{\eta}_s \text{ for }s\geq r.
\end{equation}
\item\label{Moore1-PP-odd} $[P^{n+1}(p^r),P^{n+1}(p^s)]\cong\zp{m_r^s}\langle B(\chi^r_s)\rangle$ for odd primes $p$.
   \end{enumerate}

\end{lemma}

\begin{lemma}\label{lem:pi7(P(pr))} Let $p$ be a prime number and denote $\Z/2^0=\Z/\Z=\{0\}$.
 \begin{enumerate}[(1)]
	\itemsep=1ex
	\item $\pi_7(P^{5}(2^r))\cong\Z/2^{r+1}\lara{i_4^r\nu_4}\oplus \Z/2\lara{\tilde{\eta}_r\eta}\oplus\Z/2^{m_{r-1}^2}\lara{i_4^r\Sigma \nu'}$.
	\item $\pi_7(P^{5}(p^r))\cong\left\{\begin{array}{ll}
		\Z/3^r\lara{i_4^r\nu_4}\oplus \Z/3\lara{i_4^r\Sigma \nu'},&p=3;\\
		\Z/p^{r}\lara{i_4^r\nu_4},&p\geq 5.
	\end{array}\right.$

	\noindent For $n\geq 8$,
	\item $\pi_{n+2}(P^{n}(2^r))\cong\Z/2^{m_{r}^3}\lara{i_n^r\nu}\oplus \Z/2\lara{\tilde{\eta}_r\eta}$ ;
	
	\item $\pi_{n+2}(P^n(p^r))\cong\left\{
	\begin{array}{ll}
		\Z/3\lara{i_n^r\nu}=\Z/3\lara{i_n^r\alpha_1(n)}, & \hbox{$p=3$;} \\
		0,& \hbox{$p\geq 5$.}
	\end{array}
	\right.$
\end{enumerate}
\end{lemma}

\begin{proof}
	(i) is  computed  in \cite[Section 3.3]{ZP23} and  by the same method, it is easy to compute $\pi_7(P^{5}(p^r))$ in (ii) for the odd prime number $p$.  Note that  $\pi_{n+2}(P^{n}(p^r))$ ($n\geq 8$) is in stable range for any prime number $p$. Thus (iii),(iv) are easy to get by the exact sequence of stable homotopy groups induced by the cofibration sequence (\ref{Cof Moore}).
\end{proof}

\begin{lemma}\label{lemma pi(Changcplex)}
	Let $r,s\geq 1$ be integers.
	\begin{enumerate}
		\item $\pi_7(C_{\eta}^6)\cong\Z\lara{i^{\eta}_4\nu_4}\oplus\Z/6\lara{i^{\eta}_4\Sigma \nu'}$.
		\item $\pi_7(C^6_r)\cong\Z/2^{r+1}\lara{\bar{i}_4^r\nu_4}\oplus (1-\delta_r)\Z/2\lara{\bar{i}_4^r\Sigma\nu'}\oplus \Z/2\lara{\bar{i}_P^r\tilde{\eta}_r\eta}.$
		\item $\pi_7(C^{6,s})\cong\Z/2\lara{\hat{i}_5\eta^2}\oplus\Z\lara{\hat{i}^s_4\nu_4}\oplus \Z/6\lara{\hat{i}^s_4\Sigma\nu'}$.
		\item $\pi_7(C_r^{6,s})\cong\Z/2\lara{\check{i}_5^{s,r}\eta^2}\oplus\Z/2^{r+1}\lara{\check{i}_4^{s,r}\nu_4}\oplus (1-\delta_r)\Z/2\lara{\check{i}_4^{s,r}\Sigma\nu'}\oplus \Z/2\lara{\check{i}_{P}^r\tilde{\eta}_r\eta}.$
	\end{enumerate}

	For $n\geq 5$, we have the following homotopy groups
	\begin{enumerate}
		\setcounter{enumi}{4}
		\item $ \pi_{n+3}(C_{\eta}^{n+2})\cong\Z/12\lara{i^{\eta}_n\nu}\cong\Z/4\lara{3i_{n}^{\eta}\nu}\oplus \Z/3\lara{i_{n}^{\eta}\alpha_{1}(n)}$.
		\item $\pi_{n+3}(C^{n+2}_r)\cong\Z/2^{m_r^2}\lara{\bar{i}_n^r\nu}\oplus \Z/2\lara{\bar{i}_P^r\tilde{\eta}_r\eta}$.
		\item $\pi_{n+3}(C^{{n+2},s})\cong\Z/2\lara{\hat{i}^s_{n+1}\eta^2}\oplus\Z/12\lara{\hat{i}^s_n\nu}$. 
		\item $\pi_{n+3}(C_r^{{n+2},s})\cong\Z/2\lara{\check{i}_{n+1}^{s,r}\eta^2}\oplus\Z/2^{m_r^2}\lara{\check{i}_n^{s,r}\nu}\oplus \Z/2\lara{\check{i}_{P}^r\tilde{\eta}_r\eta}.$
	\end{enumerate} 	
\begin{proof}	
The $2$-local homotopy groups of \emph{elementary Chang-complexes} are obtained in \cite{JinZhu}; the computations of the $3$-local homotopy groups of these spaces are trivial. We combine these two parts to get the above total homotopy groups.
\end{proof}
\end{lemma}

\begin{lemma}\label{Lem gen of Ceta}
	For $n\geq 3$, there are maps  
	\[\bar{\zeta}_{n}\colon C^{n+2}_\eta\to S^n,\quad \tilde{\zeta}_{n+2}\colon S^{n+2}\to C^{n+2}_\eta\] 
	satisfying  the formulae 
	\[\bar{\zeta}_n i^{\eta}_n=2\cdot 1_n,\quad q_{n+2}^{\eta}\tilde{\zeta}_{n+2}=2\cdot 1_{n+2},\] where
	 $1_k$ denotes the identity on $S^k$.
	 \begin{proof}
		See  \cite[(8.3)]{Toda2} or \cite[Theorem 3.1]{lipc22}.
	\end{proof}
\end{lemma}

\begin{lemma}\label{Lem:suspension inu4}
Let $r,s$ be positive integers.
	\begin{enumerate}
		\itemsep=1ex
		\item  If for some $a\in \Z, b\in \Z/p^r$ with $p$ an odd prime, the maps $ai_4^{\eta}\nu_4\in \pi_7(C_{\eta}^6)$,  $a\hat{i}_4\nu_4\in \pi_7(C^{6,s})$ and $bi_4^r\nu_4\in \pi_7(P^5(p^r))$ are suspensions, then $a,b=0$.
 	\item  If for some  $a_r\in \Z/2^{r+1}$, the maps $a_r i_4^r\nu_4\in \pi_7(P^5(2^r))$,  $a_r\bar{i}_4^r\nu_4\in \pi_7(C_r^6)$ and $a_r\check{i}_4^{s,r}\nu_4\in \pi_7(C_r^{6,s})$ are suspensions, then $2^r| a_r$.
 	\end{enumerate}
\end{lemma}

\begin{proof}
 Let $H:\pi_n(\Sigma X)\to  \pi_n(\Sigma X\wedge X)$ be the second James-Hopf invariant. 
By the EHP sequence, if $X$ is $(m-1)$-connected, then for $n\leq 3m-1$, 
$\alpha \in  \pi_n(\Sigma X)$ is  a suspension if and only if $H(\alpha)= 0$. 

By the	Proposition 2.2 of \cite{Todabook}, we get 
\begin{align}
	H(i_4\nu_4)=(i_4\wedge i_3)H(\nu_4)=i_4\wedge i_3\in \pi_7(X\wedge \Sigma^{-1}X) \label{equ:H(i4nu4)}
\end{align}
where $\Sigma^{-1}X$ is the desuspension of $X$,  $i_n=i_n^{\eta}, \hat{i}^s_n, i_n^r, \bar{i}_n^r, \check{i}_n^{s,r}$ are the canonical inclusions $S^n\to X$ for
 $X=C_{\eta}^{n+2}$,  $C^{n+2,s}$, $P^{n+1}(p^r)$, $C_r^{n+2,s}$,   $C_r^{n+2,s}$ respectively. 

By the Hurewicz isomorphism, we have 
\begin{align}
&\pi_7(P^5(p^r)\wedge P^4(p^r))\cong\Z/p^r\lara{i^r_4\wedge i^r_3} ~\text{($p$ is odd)};\nonumber\\ &\pi_7(X\wedge \Sigma^{-1}X)\cong\Z\lara{i_4\wedge i_3},~~ X=C_{\eta}^6,  C^{6,s};\nonumber\\
& \pi_7(X\wedge \Sigma^{-1}X)\cong\Z/2^{r}\lara{i_4\wedge i_3},~~ X=P^{6}(2^r), C_r^{6,s}, C_r^{6,s}. \label{equ:pi7(Xwedge X)}
\end{align}
From (\ref{equ:H(i4nu4)}) and (\ref{equ:pi7(Xwedge X)}), we complete the proof of this lemma now.
\end{proof}

\section{Some lemmas on cohomology operations}\label{sec:coho-operat}

\begin{lemma}\label{lem:Cohom oper +}
	Let $\Sigma X\xra{\alpha} Y\xra{i_{\alpha}} \bm C_{\alpha} \xra{q_{\alpha}} \Sigma^{2} X$  be cofibration sequences  for $\alpha=f, g$ respectively. Let $\mathcal{A}(-)\colon H^{*}( -; \Z/p)\to   H^{*+t}( -; \Z/p)  $ be a cohomology operation.  Assume that the following conditions hold for $\alpha=f,g$:
	\begin{enumerate}
		\item there are isomorphisms 
	\[ H^{k}(\bm C_{\alpha};\Z/2)\xrightarrow[\cong ]{i^{\ast}_{\alpha }}H^{k}( Y;\Z/2),\quad  H^{k+t}(\Sigma^2 X;\Z/2)\xrightarrow[\cong ]{q^{\ast}_{\alpha }}H^{k+t}( \bm C_{\alpha};\Z/2);\]
	\item $\mathcal{A}(\bm C_{g})\colon H^{k}( \bm C_{g}; \Z/p)\to   H^{k+t}( \bm C_{g}; \Z/p)  $ is trivial.
	\end{enumerate}
	Then  $\mathcal{A}(\bm C_{f+g})\colon H^{k}( \bm C_{f+g}; \Z/p)\to   H^{k+t}( \bm C_{f+g}; \Z/p)  $ is non-trivial if and only if  $\mathcal{A}(\bm C_{f})\colon H^{k}( \bm C_{f}; \Z/p)\to   H^{k+t}( \bm C_{f}; \Z/p)$ is non-trivial.
	\begin{proof}
		Consider the following homotopy commutative diagram of cofibration sequences
			\begin{align*}
			\xymatrix{
				\Sigma  X\ar[d]^{\matwo{id}{id}} \ar[r]^-{f+g} & Y \ar[r]\ar@{=}[d] & \bm C_{f+g}\ar[d]^-{\mu_{+}}\\
				\Sigma  X\vee 	\Sigma  X \ar[r]_-{(f,g)} & Y \ar[r] &    \bm C_{(f,g)} \\
				\Sigma  X\ar[u]_{M_{\alpha}} \ar[r]_-{\alpha} & Y\ar@{=}[u]  \ar[r] &    \bm C_{\alpha}\ar[u]_{\mu_{\alpha}} \\
			}.
		\end{align*}
		where $M_{\alpha}=\matwo{id}{0}, \matwo{0}{id}$ for $\alpha=f,g$ respectively.  
	There is an induced commutative diagram 
		\begin{align*}
		\xymatrix{
		 H^{k}(\bm C_{f+g};\Z/2)\ar[r]^-{\mathcal{A}(\bm C_{f+g})} &  H^{k+t}(\bm C_{f+g};\Z/2)& H^{k+t}(\Sigma^2 X;\Z/2)\ar[l]^{\cong }_-{q^{\ast}_{f+g}}\\
		 H^{k}(\bm C_{(f,g)};\Z/2)\ar[r]^-{\mathcal{A}(\bm C_{(f,g)})}\ar[u]^{\mu_{+}^{\ast}}_{\cong}\ar[d]_{u_{\alpha}^{\ast}}^-{\cong} & 	H^{k+t}(\bm C_{(f,g)};\Z/2)\ar[d]^-{u^{\ast}_{\alpha}}\ar[u]_{\mu_{+}^{\ast}} &H^{k+t}(\Sigma^2 X\!\vee\! \Sigma^2 X;\Z/2)\ar[l]^-{\cong }_-{q^{\ast}_{(f,g)}}\ar[u]^{\matwo{id}{id}^{\ast}}\ar[d]_{M^{\ast}_{\alpha}}\\
		H^{k}(\bm C_{\alpha};\Z/2)\ar[r]^-{\mathcal{A}(\bm C_{\alpha})} &   H^{k+t}(\bm C_{\alpha};\Z/2) &H^{k+t}(\Sigma^2 X;\Z/2)\ar[l]^-{\cong }_-{q^{\ast}_{\alpha}} \\
		}.
	\end{align*}
	where  $\mu_{+}^{\ast}$ and  $u^{\ast}_{\alpha}$  in the left column are isomorphisms since 	$H^{k}(\bm C_{\alpha};\Z/2)\xrightarrow {i_{\alpha }}H^{k}( Y;\Z/2)$ is isomorphic and $q^{\ast}_{f+g}, q^{\ast}_{(f,g)}, q^{\ast}_{\alpha}$ are isomorphisms by the isomorphism $H^{k+t}(\Sigma^2 X;\Z/2)\xrightarrow {q_{\alpha }}H^{k+t}( \bm C_{\alpha};\Z/2)$ .
	
	Note that $\matwo{id}{id}= M_f+M_g$. Then we have the following sequences of equivalences
	\begin{align*}
	\mathcal{A}(\bm C_{f+g})\neq 0& \iff  (q^{\ast}_{f+g})^{-1}\mathcal{A}(\bm C_{f+g})\neq 0  \iff  \matwo{id}{id}^{\ast} (q^{\ast}_{f+g})^{-1}\mathcal{A}(\bm C_{(f,g)})\neq 0 \\
	&\iff M_f^{\ast}(q^{\ast}_{f+g})^{-1}\mathcal{A}(\bm C_{(f,g)})+M_g^{\ast}(q^{\ast}_{f+g})^{-1}\mathcal{A}(\bm C_{(f,g)})\neq 0\\
	& \iff   (q^{\ast}_{f})^{-1}\mathcal{A}(\bm C_{f})+(q^{\ast}_{g})^{-1}\mathcal{A}(\bm C_{g})\neq 0\iff\mathcal{A}(\bm C_{f})\neq 0,
	\end{align*}
	where the last equivalence holds by the assumptions in the Lemma.
	\end{proof}
\end{lemma}

Consider the following complexes:
\begin{align*}	
	 A^{s}(k,k')=P^{7}(2^s)\cup_{ki_{6}^s\eta^2+k'\tilde{\eta}_s}e^{9};
	\quad &A(k)=S^6\cup_{k\eta^2}e^{9}; \quad C^{\eta}(t)=C_{\eta}^7\cup_{ti_{5}^{\eta}\nu}e^9;\\
	 A_{r}(k,t)=P^{6}(2^r)\cup_{k\tilde{\eta}_s\eta+ti_{n}^r\nu}e^{9};\quad &\bar{C}(k,t)=C_{r}^{7}\cup_{k\bar{i}_{P}^r\tilde{\eta}_s\eta+t\bar{i}_{5}^r\nu}e^{9}; \\
	\hat{C}(k,t)=C^{7,s}\cup_{k\hat{i}^s_6\eta^2+t\hat{i}^s_5\nu}e^{9}; \quad & \check{C}(k,k',t)=C_r^{7,s}\cup_{k\check{i}_6^{r,s}\eta^2+k'\check{i}_{P}^r\tilde{\eta}_s\eta+t\check{i}_5^{r,s}\nu}e^{9}.
\end{align*}
where  $k,k'\in \Z/2$, $t$ is an element in  $\Z/12$ for $\hat{C}(k,t)$, $\Z/2^{m_{r}^3}$ for $A_{r}(k,t)$, $\Z/2^{m_{r}^2}$ for $\bar{C}(k,t)$ and $\check{C}(k,k')$.

 \begin{lemma}\label{lem:Sq2}
  	The Steenrod square  
  \[\Sq^2\colon H^{7}( A^{s}(k,k');\Z/2)\to H^{9}( A^{s}(k,k');\Z/2)\] 
   is non-trivial if and only if $k'=1$.
 \end{lemma}
 \begin{proof}
	From Section 2 of \cite{lipc23}, we get  $\Sq^2\colon H^{7}( A^{s}(0,k');\Z/2)\to H^{9}( A^{s}(0,k');\Z/2)$ is isomorphic if and only if $k'=1$. Clearly, $\Sq^2$ is trivial on the cohomologies of $A^{s}(k,0)$. So we get this lemma by Lemma \ref{lem:Cohom oper +}.
\end{proof}

\begin{lemma}\label{lem:Theta}
~~

\begin{enumerate}[(1)]
	\item \label{case Theta on X}
$\Theta\colon H^6(X;\Z/2)\to H^{9}(X;\Z/2)$ for $X=A^{s}(k,k'),A(k),$ $ A_{r}(k,t),\bar{C}(k,t),\hat{C}(k,t)$ is non-trivial if and only if $k=1$;
\item  \label{case Theta on C(k,k',t)}
$\Theta\colon H^6(\check{C}(k,k',t);\Z/2)\to H^{9}(\check{C}(k,k',t);\Z/2)$ for is non-trivial if and only if $k=1$ or $k'=1$.
\end{enumerate}
\begin{proof}
 It is clear that 
$ S_6(X)=H^6(X;\z{}),\quad T_6(X)=H^{9}(X;\z{}) $ for all $X$ given in   (\ref{case Theta on X}) and $X=\check{C}(k,k',t)$ in (\ref{case Theta on C(k,k',t)}).
We only prove the (\ref{case Theta on C(k,k',t)}) of this lemma since the proof of (\ref{case Theta on X}) is similar and easier.
	
There are homotopy commutative diagrams of cofibration sequences 
		\begin{align*}
		\xymatrix{
			S^6\ar@{=}[d] \ar[r]^-{\smatwo{2^s}{\eta}} & S^{6}\vee S^5 \ar[r]^-{\hat{i}^{s}}\ar[d]^-{p_1^{6}} & C^{7,s}\ar[d]^-{\hat{q}_{P}} \\
			S^6\ar[r]^-{2^{s}} & S^6  \ar[r]^-{i_{6}^{s}} & P^7(2^{s});
		} \qquad	\xymatrix{
		S^8\ar@{=}[d] \ar[r]^-{k\hat{i}^{s}_6\eta^2} &  C^{7,s} \ar[r]\ar[d]^-{\hat{q}_{P}} & \hat{C}(k,0)\ar[d]^-{ q_A} \\
		S^8\ar[r]^-{ki_6^s\eta^2} &P^7(2^{s}) \ar[r]&A^s(k,0).
		}
	\end{align*}
	 By the naturality of $\Theta$, we get the  following  commutative diagram of mod $2$ cohomology groups (the coefficients $\z{}$ are omitted)
    \[\begin{tikzcd}[sep=scriptsize]
   	H^6(A^s(k,0))\ar[r," q_A^\ast","\cong"swap]\ar[d,"\Theta",swap]& H^6(\hat{C}(k,0))\ar[d,"\Theta"]\\
   	H^{9}(A^s(k,0))\ar[r," q_A^\ast","\cong"swap]& H^{9}(\hat{C}(k,0))
   \end{tikzcd},\]
 where the isomorphisms $ q_A^\ast$ can be easily checked. 
 
	From \cite[Lemma 2.7]{lipc23}, we get that  the left $\Theta$ above is non-trivial if and only if $k=1$. So
	\begin{align}
		\Theta:H^6(\hat{C}(k,0))\to   H^{9}(\hat{C}(k,0))~\text{is isomorphic~}\iff k=1. \label{iff Theta hatC(k,0)}
	\end{align}
	By the diagram (\ref{Diagram 2 for <<}), $\check i_{C}^s \hat i^s_6=\check{i}_6^{s,r}$. So there are the following homotopy commutative diagrams of cofibration sequences
		\begin{align}
	\xymatrix{
		S^8\ar@{=}[d] \ar[r]^-{k\hat{i}^{s}_6\eta^2} &  C^{7,s} \ar[r]\ar[d]^-{\check{i}_{C}^s} & \hat{C}(k,0)\ar[d]^-{\check{i}_{C}} \\
		S^8\ar@{=}[d]\ar[r]^-{k\check{i}_6^{s,r}\eta^2} &C_r^{7,s}\ar[d]^-{p_2^6\check q_{r}^s} \ar[r]&\check{C}(k,0,0)\ar[d]^-{\check q_{S}}\\
			S^8 \ar[r]^-{0} &S^6 \ar[r]&S^6\vee S^9
	}; \qquad	\xymatrix{
			S^8\ar@{=}[d] \ar[r]^-{k'\tilde{\eta}_r\eta} &  P^{6}(2^r) \ar[r]\ar[d]^-{\check i_{P}^r} & A_r(k',0)\ar[d]^-{\check{i}_{A}} \\
			S^8\ar@{=}[d]\ar[r]^-{k'\check i_{P}^r\tilde{\eta}_r\eta} &C_r^{7,s}\ar[d]^-{\check q_{P}^s} \ar[r]&\check{C}(0,k',0)\ar[d]^{\check q_P}\\
				S^8 \ar[r]^-{0} &P^7(2^{s}) \ar[r]&P^7(2^{s}) \vee S^9
		}. \label{Diagram:Ar to C(0,k',0)}
	\end{align}
	From the left diagram above we get the following commutative diagram 
	 \[\begin{tikzcd}[sep=scriptsize]
	 H^6(S^6\vee S^9)\ar[r," \check q_S^\ast", "\text{inject.}"swap]\ar[d,"\Theta",swap]& H^6(\check{C}(k,0,0))\cong \Z/2\oplus\Z/2 \ar[d,"\Theta"]\ar[r," \check i_C^\ast","\text{surject.}"swap ]&H^6(\hat{C}(k,0))\cong \Z/2\ar[d,"\Theta"]\\
	 H^{9}(S^6\vee S^9)\ar[r,"  \check q_S^\ast","\cong"swap]& H^{9}(\check{C}(k,0,0))\cong \Z/2\ar[r," \check i_C^\ast","\cong"swap]& H^6(\hat{C}(k,0))\cong \Z/2
	\end{tikzcd}.\]
	It is easy to check the top row is a split short exact sequence,  $\check q_S^\ast$ and $\check i_C^\ast$ in the bottom row are isomorphic.
  Moreover, $\Theta=0\colon H^{6}(S^6\vee S^9)\to  H^{9}(S^6\vee S^9)$. So the middle $\Theta$ above is isomorphic if and only if the right $\Theta$ is isomorphic.  Hence from (\ref{iff Theta hatC(k,0)}) we  get 
		\begin{align}
		\Theta\colon H^6(\check{C}(k,0,0))\to   H^{9}(\check{C}(k,0,0))~\text{is isomorphic~}\iff k=1. \label{iff Theta checkC(k,0,0)}
	\end{align}

 There is a map $A_r(k',0)\xra{q_{k'\eta^2}} \bm C_{k'\eta^2}$ satisfying  the 
  following left homotopy commutative diagram of cofibration sequences and $q_{k'\eta^2}$ induces the following right commutative diagram with isomorphisms $q^{\ast}_{k'\eta^2}$
  \begin{align*}
  		\xymatrix{
  		S^8\ar@{=}[d] \ar[r]^-{k'\tilde{\eta}_r\eta} &  P^{6}(2^r) \ar[r]\ar[d]^-{q_{6}^r} & A_r(k',0)\ar[d]^-{q_{k'\eta^2}} \\
  		S^8\ar[r]^-{k'\eta^2} &S^6 \ar[r]&\bm C_{k'\eta^2},
  	}\qquad  	\xymatrix{
  	H^6( \bm C_{k'\eta^2})\ar[d]^-{\Theta} \ar[r]_-{\cong}^-{q^{\ast}_{k'\eta^2}} & 	H^6( A_r(k',0))\ar[d]^-{\Theta}  \\
  	H^9( \bm C_{k'\eta^2}) \ar[r]_-{\cong}^-{q^{\ast}_{k'\eta^2}}&	H^9( A_r(k',0)).
  	}
  \end{align*}
  Also from the \cite[Lemma 2.7]{lipc23}, we get that  the left $\Theta$ above is non-trivial if and only if $k=1$. So is the same fact about the  right $\Theta$ above. 
  Now considering the diagram (\ref{Diagram:Ar to C(0,k',0)}), by the similar proof as that of (\ref{iff Theta checkC(k,0,0)}), we have 
	\begin{align}
	\Theta\colon H^6(\check{C}(0,k',0))\xra{\cong}  H^{9}(\check{C}(0,k',0))\iff k=1. \label{iff Theta checkC(0,k',0)}
\end{align}
	 $\check i_{P}^r\tilde{\eta}_r\eta\prec\check{i}_6^{s,r}\eta^2$ from Lemma \ref{lem:order pi6(X)} implies that   $\check{C}(1,1,0)=\check{C}(0,1,0)$. So by (\ref{iff Theta checkC(k,0,0)}) and (\ref{iff Theta checkC(0,k',0)}) we get
	\begin{align}
		\Theta\colon H^6(\check{C}(k,k',0))\xra{\cong}  H^{9}(\check{C}(k,k',0))\iff k=1~\text{or}~k'=1. \label{iff Theta hatC(k,k',0)}
	\end{align}
	Moreover, $\Theta=0\colon H^6(\check{C}(0,0,t))\to   H^{9}(\check{C}(0,0,t))$ 
	by the following (homotopy) commutative diagrams
	\begin{align*}
		\xymatrix{
			S^8\ar@{=}[d] \ar[r]^-{t\nu} &  S^{5} \ar[r]\ar[d]^-{\check{i}_5^{s,r}} & \bm C_{t\nu}\ar[d]^-{i_{t\nu}} \\
			S^8\ar[r]^-{t\check{i}_5^{s,r}\nu} &C_{r}^{7,s} \ar[r]&\check{C}(0,0,t),
		}\qquad  	\xymatrix{
			H^6( \check{C}(0,0,t))\ar[d]^-{\Theta} \ar[r]^-{i^{\ast }_{t\nu}} & 	H^6(  \bm C_{t\nu})=0\ar[d]^-{\Theta}  \\
			H^9( \check{C}(0,0,t)) \ar[r]_-{\cong}^-{i^{\ast }_{t\nu}}&	H^9(  \bm C_{t\nu}).
		}
	\end{align*}
	Now from Lemma \ref{lem:Cohom oper +} and (\ref{iff Theta hatC(k,k',0)}), we complete the proof of (\ref{case Theta on C(k,k',t)}) of Lemma \ref{lem:Theta}.
   \end{proof}
  \end{lemma}

\begin{lemma}\label{lem:Adams}
	Let $X=\bm C_{t\nu}=S^5\cup_{t\nu}e^9$, $A_r(k,t)$, $C^{\eta}(t)$, $\bar{C}(k,t)$, $\hat C(k,t)$ or $\check{C}(k,k',t)$.
Then the Adem operation 
\[\Psi\colon H^5(X;\z{})\to H^9(X;\z{})\] 
is an isomorphism if $t\equiv 2 \pmod 4$  and is trivial if $t\equiv 0 \pmod 4$. 
\end{lemma}

 \begin{proof}
 The lemma is true for $X=\bm C_{t\nu}$ by \cite[Lemma 4.4.4]{Adams60}. Next we only show this for $X=\check{C}(k,k',t)$ since the other cases are easier. 
 
 The lemma is also true for $X=\check{C}(0,0,t)$  by the naturality of $\Psi$ if we consider the map $\bm C_{t\nu}\xra{J}  \check{C}(0,0,t)$ induced by the following homotopy commutative diagram 
 	\begin{align*}
 	\xymatrix{
 		S^8\ar@{=}[d] \ar[r]^-{t\nu} & S^{5} \ar[r]\ar[d]^-{\check{i}_5^{s,r}} & \bm C_{t\nu}\ar[d]^-{J} \\
 		S^8\ar[r]^-{t\check{i}_5^{s,r}\nu} & C_{r}^{7,s}  \ar[r] &    \check{C}(0,0,t).
 	}
 \end{align*}
 
 There are the following homotopy commutative diagrams with cofibre sequence rows
 \begin{equation} \label{diagram:natural for Phi}
	\begin{aligned}
		\xymatrix{
			S^8\ar@{=}[d] \ar[r]^-{k\eta^2} & S^{6} \ar[r]\ar[d]^-{\check{i}_6^{s,r}} & \bm C_{k\eta^2}\ar[d]^-{J_1} \\
			S^8\ar[r]^-{k\check{i}_6^{s,r}\eta^2} & C_{r}^{7,s}  \ar[r] &     \check{C}(k,0,0)
		}; 	\xymatrix{
		S^8\ar@{=}[d] \ar[r]^-{k'\tilde{\eta}_1\eta} & P^{6}(2) \ar[r]\ar[d]^-{\check{i}_P^rB(\chi_r^1)} & \bm C_{k'\tilde{\eta}_1\eta}\ar[d]^-{J_2} \\
		S^8\ar[r]^-{k'\check{i}_P^r\tilde{\eta}_r\eta} & C_{r}^{7,s}  \ar[r] &   \check{C}(0,k',0)
		}.
	\end{aligned}
 \end{equation}	
  By the definition of $\Psi$, it is defined on the elements $u\in H^{5}(-;\Z/2)$ such that $Sq^i(u)=0$ for $i=1,3,4$.
 
 Since $H^5(\bm C_{k\eta^2};\Z/2)=0$ and $Sq^1\colon H^5(\bm C_{k'\tilde{\eta}_1\eta})\to  H^6(\bm C_{k'\tilde{\eta}_1\eta})$ is isomorphic,  by  commutative diagrams (\ref{diagram:natural for Phi}) and the naturality of $\Psi$, we get 
 \begin{align}
 \Psi=0 \colon H^5(X;\z{})\to H^9(X;\z{}), ~~ X=\check{C}(k,0,0),\check{C}(0,k',0)\label{trivial Psi}
 \end{align}
So the Lemma \ref{lem:Adams} is also ture for $X=\check{C}(k,k',t)$ by Lemma \ref{lem:Cohom oper +}.
 \end{proof}

  \begin{lemma}\label{lem:p-odd:n=3:P1}
    The first Steenrod power $\PP$ acts non-trivially on $H^5(X;\Z/3)$ for $X=S^5\cup_{\alpha_1(5)}e^9$, $P^{6}(3^r)\cup_{i^r_5\alpha_1(5)}e^9$, $C_{\eta}^7\cup_{i^{\eta}_5\alpha_1(5)}e^9$, $C^{7,s}\cup_{\hat i^{s}_5\alpha_1(5)}e^9$.
  
   \begin{proof}
     It is well-known that $\PP$ is non-trivial for $X=S^5\cup_{\alpha_1(5)}e^9$ (cf. \cite[Chapter 1.5.5]{Harperbook}).  For other complexes in the Lemma, $\PP\neq 0$ can be easily deduced from the naturality of $\PP$.
   \end{proof}
  \end{lemma}

 \section{Homotopy decompositions of $\Sigma^3M$}\label{sect:proofs}
 Let $M$ be a closed oriented simply-connected 6-manifold with $H_\ast(M)$ given by (\ref{HM}). Using Wall's well-known splitting theorem for smooth or topological $6$-manifolds \cite{Wall66,Jupp73},  Huang \cite[Corollary 2.2]{Huang-6mflds} proved the following homotopy decomposition 
 \begin{align}
  \Sigma M\simeq\Sigma M'\vee dS^4, \label{equ:M=M'vee dS4}
 \end{align}
where $M'$ is a $6$-manifold with $H_\ast(M')$ given by
\begin{equation}\label{HM'}
	\begin{tabular}{cccccccc}
		\toprule
		$i$& $2$&$3$&$4$&$0,6$&$\text{otherwise}$
		\\\midrule
		$H_i(M')$& $\Z^l\oplus T$&  $T$ & $\Z^l$& $\Z$&$0$
		\\ \bottomrule
	\end{tabular}.
\end{equation}
Let $e^6$ be the top cell of $M'$ and denote by $\ov{M'}$ the complement of $e^6$ in $M'$. Note that the suspension space $\Sigma \ov{M'}$ is an $\mathbf{A}_3^2$-complex. Using the method of undeterminated coefficients and the indecomposable homotopy types of $\mathbf{A}_{3}^2$-complexes given in Section \ref{sect:intro},  we have a homotopy equivalence 
\begin{align*}
	\Sigma \ov{M'}\simeq &P^{4}(T_{\geq 3}) \vee P^{5}(T_{\geq 3}) \vee (l-k-t_3)S^{3}\vee k C_{\eta}^{5}\vee (l-k-t_2) S^{5} \\
	&\vee (\bigvee_{j=1}^{t_0}P^{5}(2^{s_j}))\vee(\bigvee_{j=1}^{t_1}P^{4}(2^{r_j})) \vee (\bigvee_{j=1}^{t_2}C_{\bar{r}_{j}}^{5})\vee (\bigvee_{j=1}^{t_3}C^{5,\hat{s}_j})\vee (\bigvee_{j=1}^{t_4}C_{\check{r}_{j}}^{5,\check{s}_j}),
\end{align*}
where $0\leq k+t_2, k+t_3\leq l$; $r_j (0\leq j\leq t_1)$, $\bar{r}_j (0\leq j\leq t_2)$,  $\check{r}_j (0\leq j\leq t_4)$ and $s_j (0\leq j\leq t_0)$, $\hat{s}_j (0\leq j\leq t_3)$,  $\check{s}_j (0\leq j\leq t_4)$ satisfy the equations 
\begin{align}
	&\bigoplus_{j=1}^{t_1}\z{r_j}\oplus \bigoplus_{j=1}^{t_2}\z{\bar{r}_j}\oplus \bigoplus_{j=1}^{t_4}\z{\check{r}_j}=T_2,~~t_1+t_2+t_4=m_2; \label{equ:ri} \\
	&\bigoplus_{j=1}^{t_0}\z{s_j}\oplus\bigoplus_{j=1}^{t_3}\z{\hat{s}_j}\oplus \bigoplus_{j=1}^{t_4}\z{\check{s}_j}=T_2,~~t_0+t_3+t_4=m_2. \label{equ:si} 
\end{align}

By the work of Culter and So \cite{CS22}, there is a homotopy equivalence 
\begin{equation}\label{eq:SM'}
	\Sigma M'\simeq P^4(T_{\geq 5})\vee P^5(T_{\geq 3})\vee \bm C_{h'},	
\end{equation}
where $\bm C_{h'}$ is the homotopy cofibre of some map $h'\colon S^6\to V_5'$ with 
\begin{equation}\label{eq:V_5}
	\begin{aligned}
	V_5'=&(\bigvee_{j=1}^{m_3}P^{4}(3^{r'_j}))\vee(l-k-t_3)S^{3}\vee k C_{\eta}^{5}\vee (l-k-t_2) S^{5} \vee (\bigvee_{j=1}^{t_0}P^{5}(2^{s_j}))\\
	&\vee(\bigvee_{j=1}^{t_1}P^{4}(2^{r_j}))  \vee (\bigvee_{j=1}^{t_2}C_{\bar{r}_{j}}^{5})\vee (\bigvee_{j=1}^{t_3}C^{5,\hat{s}_j})\vee (\bigvee_{j=1}^{t_4}C_{\check{r}_{j}}^{5,\check{s}_j}).
\end{aligned} 
\end{equation}
 Moreover,  \cite[Lemma 3.3]{CS22} implies that the composite 
 \[S^6\xra{h'} V_5'\to (\bigvee_{j=1}^{m_3}P^{4}(3^{r'_j}))\to P^4(3^{r'_j})\] 
 is generated by $i_3^{r_j'}\alpha_1(3)$, where the unlabeled maps are the obvious pinch maps.
 Let $Z$ be any a wedge summand  of $\Sigma V'_5$ and let $p_{Z}\colon \Sigma V'_5 \to Z$ be the canonical projection, then the composite  $S^7\xra{\Sigma h'}\Sigma V'_5 \xra{p_Z} Z$ is a suspension. Suspending (\ref{eq:SM'}) we have a homotopy equivalence
\begin{equation}\label{eq:S2M'}
	\Sigma^2 M'\simeq P^{6}(T_{\geq 3})\vee P^{5}(T_{\geq 5})\vee \bm C_{h''},
\end{equation}
where $h''=\Sigma h'\colon S^7\to  V''_6$ with
\begin{multline*}
	V''_6=\Sigma V'_5=(\bigvee_{j=1}^{m_3}P^{5}(3^{r'_j}))\vee (l-k-t_3)S^{4}\vee k C_{\eta}^{6}\vee (l-k-t_2) S^{6}\vee (\bigvee_{j=1}^{t_0}P^{6}(2^{s_j}))\\
	\vee(\bigvee_{j=1}^{t_1}P^{5}(2^{r_j}))  \vee (\bigvee_{j=1}^{t_2}C_{\bar{r}_{j}}^{6})\vee (\bigvee_{j=1}^{t_3}C^{6,\hat{s}_j})\vee (\bigvee_{j=1}^{t_4}C_{\check{r}_{j}}^{6,\check{s}_j}).
\end{multline*}
 Using the homotopy groups computed in Section \ref{sec:An2}, we may put 
 \begin{align*}
	h''=&\sum_{j=1}^{m_3} a'_j\cdot i_4^{r'_j}\alpha_{1}(4)+\sum_{j=1}^{l-k-t_3} b^S_j\cdot \Sigma \nu'+\sum_{j=1}^{k} b^{\eta}_j\cdot i_4^{\eta}\Sigma \nu'+\sum_{j=1}^{l-k-t_2} d^S_j\cdot \eta\\ 
	&+\sum_{j=1}^{t_0} \big(x_j\cdot \tilde{\eta}_{s_j}+y_j\cdot i_{5}^{s_j}\eta^2\big)+\sum_{j=1}^{t_1}\big(2^{r_j}k_j\cdot i_4^{r_j}\nu_4+a_j\cdot \tilde{\eta}_{r_j}\eta+b_j\cdot i_{4}^{r_j}\Sigma\nu'\big)\\
	&+\sum_{j=1}^{t_2}\big(2^{\bar{r}_j}\bar{k}_j\cdot i_4^{\bar{r}_j}\nu_4+\bar{a}_j\cdot \bar{i}_{P}^{\bar{r}_j}\tilde{\eta}_{\bar{r}_j}\eta+\bar{b}_j\cdot \bar{i}_{4}^{\bar{r}_j}\Sigma\nu'\big)+\sum_{j=1}^{t_3}\big(\hat{b}_j\cdot \hat{i}^{\hat{s}_{j}}_4\Sigma \nu'+\hat{y}_j\cdot \hat{i}^{\hat{s}_j}_5\eta^2\big)\\
	&+\sum_{j=1}^{t_4}\big(2^{\check{r}_j}\check{k}_j\cdot \check{i}^{\check{s}_j,\check{r}_j}_4\nu_4+\check{a}_j\cdot \check{i}_{P}^{\check{r}_j}\tilde{\eta}_{\check{r}_j}\eta+\check{b}_j\cdot \check{i}^{\check{s}_j,\check{r}_j}_4\Sigma\nu'+\check{y}_j\cdot \check{i}^{\check{s}_j,\check{r}_j}_5\eta^2\big),
\end{align*}
where $k_j,\bar{k_j},\check{k_j}=0$ or $1$; $a'_j\in\Z/3$; $b_j^S\in \Z/12$; $b_j^{\eta}, \hat{b}_j\in \Z/6$; $x_j\in \Z/4$ and $\Z/2$ for $s_j=1$ and $s_j\geq 2$ respectively;
$y_j\in (1-\delta_{s_j})\Z/2$ ($s_j=1$ applies $y_j=0$); $d_j^S, a_j, \bar{a}_j,  \hat{y}_j, \check{a}_j,  \check{y}_j\in \Z/2$; $b_j\in \Z/2^{m_{r_j}^2}$; $\bar{b}_j\in (1-\delta_{\bar{r}_j})\Z/2$; $\check{b}_j\in (1-\delta_{\check{r}_j})\Z/2 $.

The discussion of the homotopy types of $\bm C_{h''}$ is complicated because of the linearly independence of $\nu_4,\Sigma\nu'$. After one more suspension, there holds $\Sigma^2\nu'=2\nu$ and the expression for $\Sigma h''$ has a simpler form: 
 \begin{align*}
\Sigma h''	=	&\sum_{j=1}^{m_3} a'_j\cdot i_5^{r'_j}\alpha_{1}(5)+\sum_{j=1}^{l-k-t_3} 2b^S_j\cdot \nu+\sum_{j=1}^{k} 2b^{\eta}_j\cdot i_5^{\eta} \nu+\sum_{j=1}^{l-k-t_2} d^S_j\cdot \eta\\
	&+\sum_{j=1}^{t_0}\big(x_j\cdot \tilde{\eta}_{s_j}+y_j\cdot i_{6}^{s_j}\eta^2\big)+\sum_{j=1}^{t_1}\big(a_j\cdot \tilde{\eta}_{r_j}\eta+2b_j\cdot i_{5}^{r_j}\nu\big)\\
	& +\sum_{j=1}^{t_2}\big(\bar{a}_j\cdot \bar{i}_{P}^{\bar{r}_j}\tilde{\eta}_{\bar{r}_j}\eta+2\bar{b}_j\cdot \bar{i}_{5}^{\bar{r}_j}\nu\big)+\sum_{j=1}^{t_3}\big(2\hat{b}_j\cdot\hat{i}^{\hat{s}_{j}}_5 \nu+\hat{y}_j\cdot \hat{i}^{\hat{s}_j}_6\eta^2\big)
	\\
	&+\sum_{j=1}^{t_4}\big(\check{a}_j\cdot \check{i}_{P}^{\check{r}_j}\tilde{\eta}_{\check{r}_j}\eta+2\check{b}_j\cdot \check{i}^{\check{s}_j,\check{r}_j}_5\nu+\check{y}_j\cdot\check{i}^{\check{s}_j,\check{r}_j}_6\eta^2\big).
\end{align*}

From now on we assume that the Adem operation $\Psi$ acts trivially on $H^2(M;\z{})$. Then Lemma \ref{lem:Adams} implies that   $b_j^S, b_j^{\eta}, b_j, \bar{b}_j, \hat{b}_j, \check{b}_j$ must be divided by $2$ for all $j$ in their corresponding ranges. Since $4i_{5}^{r}\nu=i_{5}^{r}\eta^3$ for $r\geq 3$ and $4i_{5}^{r}\nu=0$ for $r=1,2$, while $4\bar{i}_{5}^{\bar{r}_j}\nu=4\check{i}^{\check{s}_j,\check{r}_j}_5\nu=0$, we have 
\begin{equation}\label{equ:Sigma g7} 
	 \begin{aligned}
		\Sigma h''=&\sum_{j=1}^{m_3} a'_j\cdot i_5^{r'_j}\alpha_{1}(5)+\sum_{j=1}^{l-k-t_3} 4c^S_j\cdot \nu+\sum_{j=1}^{k} 4c^{\eta}_j\cdot i_5^{\eta} \nu+\sum_{j=1}^{l-k-t_2} d^S_j\cdot \eta \\
	&+\sum_{j=1}^{t_0}\big( x_j\cdot \tilde{\eta}_{s_j}+y_j\cdot i_{6}^{s_j}\eta^2\big)+\sum_{j=1}^{t_1}\big(a_j\cdot \tilde{\eta}_{r_j}\eta+c_j\cdot i_{5}^{r_j}\eta^{3}\big)\\
	&+\sum_{j=1}^{t_2}\big(\bar{a}_j\cdot \bar{i}_{P}^{\bar{r}_j}\tilde{\eta}_{\bar{r}_j}\eta\big)+\sum_{j=1}^{t_3}\big(4\hat{c}_j\cdot \hat{i}^{\hat{s}_{j}}_5 \nu+\hat{y}_j\cdot \hat{i}^{\hat{s}_j}_6\eta^2\big)\\
	&+\sum_{j=1}^{t_4}\big(\check{a}_j\cdot  \check{i}_{P}^{\check{r}_j}\tilde{\eta}_{\check{r}_j}\eta+\check{y}_j\cdot\check{i}^{\check{s}_j,\check{r}_j}_6\eta^2\big),  
\end{aligned}
\end{equation}
where $c_j^S\in\{0,1,2,3,4,5\}$ , $c_j^{\eta}, \hat{c}_j\in \{0,1,2\}$, $c_j\in\{0,1\}$ for $r_j\geq 3$ and $c_j=0$ for $r_j=1,2$.

\subsection{Reduction of coefficients}\label{sect:reduction}
Let $\alpha\colon W\to  X$ and $\beta\colon W\to  Y$ be (based) maps, we denote $\alpha\prec \beta$ if $\beta=f\alpha$ for some map $f:X\to  Y$.  It is easy to see that $\prec$ has the transitive property: if $\alpha\prec\beta$ and $\beta\prec\gamma$ for $\gamma\colon W\to  Z$, then $\alpha\prec\gamma$.
Moreover $\alpha\prec\beta$ implies that $\matwo{\alpha}{\beta}\sim \matwo{\alpha}{0}\colon W\to  X\vee Y$.

\begin{lemma}\label{lem:order pi6(X)}
	Let $s<s'$, $r<r'$ be positive integers. For maps $S^{n+3}\to X$ with $X$ elementary $\an{2}$-complexes, $n\geq 4$, there holds a chain of relations:
	\begin{multline*}
		\tilde{\eta}_s\prec\tilde{\eta}_{s'}\prec\eta\prec\tilde{\eta}_{r}\eta\prec\check{i}^r_{P}\tilde{\eta}_{r}\eta\prec\bar{i}^r_{P}\tilde{\eta}_{r}\eta\prec \tilde{\eta}_{r'}\eta
		\prec i_{n+1}^{s'}\eta^2\\\prec\hat{i}^{s}_{n+1}\eta^2
		\prec\check{i}_{n+1}^{s,r}\eta^2\prec i_{n+1}^{s}\eta^2\prec\eta^{3}\prec i_{n}^r\eta^3.
	\end{multline*}		
\begin{proof} 
	Without loss of generality, we only prove the chain when $n=4$.

	(a) The relations $\tilde{\eta}_s\prec \tilde{\eta}_{s'}\prec\eta$ follow by the formulae 
	\[\tilde{\eta}_{s'}=B(\chi_{s'}^s)\tilde{\eta}_s,\quad q_{5}^s\tilde{\eta}_{s'}=\eta\in\pi_{7}(S^6).\] 
	
	(b) $\check{i}^r_{P}\tilde{\eta}_{r}\eta\prec \bar{i}^r_{P}\tilde{\eta}_{r}\eta$ holds for $\check{q}_C^r\check{i}_P^r=\bar{i}_P^r$. 
	
	(c) $\bar{i}^r_{P}\tilde{\eta}_{r}\eta \prec  \tilde{\eta}_{r'}\eta$: Consider the following homotopy commutative diagram that induces the map $\bar{q}^r_{r'}$:
	\begin{equation}\label{diagram qrr'}
	\begin{aligned}
		\xymatrix{
			S^4 \ar@{=}[d] \ar[r]^-{2^r} & S^4 \ar[d]^-{\hat{i}^{s}_4} \ar[r]^-{i_{4}^r} & P^5(2^r) \ar[d]^-{\check{i}^r_{P}} \ar[r]^-{q_5^r} & S^5\ar@{=}[d] \\
			S^4\ar@{=}[d] \ar[r]^-{2^r\hat{i}^{s}_4} & C^{6,s}\ar[d]^-{\hat{q}_{\eta}}  \ar[r]^-{\check{i}^s_{C}} & C^{6,s}_{r} \ar[r]\ar[d]^-{\check{q}^r_{C}} & S^5\ar@{=}[d]\\
			S^4\ar@{=}[d] \ar[r]^-{2^r i^{\eta}_{4}} & C^{6}_{\eta} \ar[r]^-{\bar{i}_{\eta}}\ar[d]^-{2^{r'-r-1}\bar{\zeta}_{4}} & C^{6}_{r} \ar[r]^-{\bar{q}_5^r}\ar[d]^-{\bar{q}_{r'}^r} & S^5\ar@{=}[d] \\
			S^4\ar[r]^-{2^{r'}} & S^4  \ar[r]^-{i_{4}^{r'}} & P^5(2^{r'})  \ar[r] & S^5
		} 
	\end{aligned}.
\end{equation}
	The compositions in the middle two columns imply that  
	\[\bar{q}_{r'}^r \check{q}^r_{C} \check{i}^r_{P}=B(\chi_{r'}^{r})+\varepsilon \cdot i^{r'}_4\eta q^r_5\text{ for some $\varepsilon\in\z{}$}. \] 
	After replacing $\bar{q}_{r'}^r$ by $\bar{q}_{r'}^r+  i^{r'}_{4}\eta \bar{q}_5^r$ if necessary, we can assume that $\varepsilon=0$.
	Then by (\ref{eq:chi-eta}) we have 
	\begin{align}
		\tilde{\eta}_{r'}\eta=B(\chi_{r'}^{r})\tilde{\eta}_{r}\eta=\bar{q}_{r'}^r \check{q}^r_{C} \check{i}^r_{P}\tilde{\eta}_{r}\eta=\bar{q}_{r'}^r\bar{i}^r_{P}\tilde{\eta}_{r}\eta. \label{equ:qrr'}
	\end{align}
	
	(d) $\tilde{\eta}_{r'}\eta\prec i^{s'}_5\eta^2$:   $i_5^{s'}q_5^{r'}\tilde{\eta}_{r'}\eta=i_5^{s'}\eta^2$.
	
	(e)  $i^{s'}_5\eta^2\prec \hat{i}^{s}_5\eta^2\prec \check{i}_5^{s,r}\eta^2\prec i_5^{s}\eta^2$:  Consider the following homotopy commutative diagram
	\begin{equation}\label{Diagram 2 for <<}
		\begin{aligned}
		\xymatrix{
			S^5\ar[d]_-{2^{s'-s-1}\bar{\xi}_{5}} \ar[r]^-{2^{s'}} & S^{5}\ar@{=}[d]  \ar[r]^-{i^{s'}_{n+1}} & P^{6}(2^{s'}) \ar[r]\ar[d]^-{\bar{\xi}_{s}^{s'}} & S^{6}\ar[d]\\
			C_{\eta}^5\ar[d]_-{\bar{i}_{\eta}} \ar[r]^-{2^sq_5^{\eta}} & S^{5}\ar@{=}[d]  \ar[r]^-{\hat{i}^s_5} & C^{6,s} \ar[r]\ar[d]^-{\check{i}^s_{C}} & C_{\eta}^6\ar[d]\\
			C_{r}^5\ar[d]_-{\bar{q}_5^r} \ar[r]^-{2^s \bar{q}_5^r} & S^{5} \ar@{=}[d] \ar[r]^-{\check{i}_5^{s,r}} & C^{6,s}_{r} \ar[r]\ar[d]^-{\check{q}_{P}^s}& C_{r}^6\ar[d] \\
			S^5\ar[r]^-{2^{s}} & S^5  \ar[r]^-{i_{5}^{s}} & P^6(2^{s})  \ar[r] & S^6
		}, 
	\end{aligned}
	\end{equation}
	where all rows are homotopy cofibre sequences,
	 $\bar{\xi}_{s}^{s'}$ is the induced map by the left-top-homotopy commutative square (by Lemma \ref{Lem gen of Ceta}) in the above diagram. 
	Then we have \begin{align*}
		i^{s'}_5\eta^2\prec \hat{i}^{s}_5\eta^2&:~~i_5^{s'}\bar{\xi}_{s'}^{s}=\hat{i}^s_5(s<s');	\\
		\hat{i}^{s}_5\eta^2\prec\check{i}^{s,r}_5\eta^2 &:~~\check{i}^{s,r}_5\eta^2=\check{i}^s_{C}\hat{i}^{s}_5\eta^2;\\
		\check{i}^{s,r}_5\eta^2\prec i_5^{s}\eta^2&:~~i_5^s\eta^2=\check{q}_{P}^s\check{i}^{s,r}_5\eta^2.
	\end{align*}

	(f)  $i_5^{s}\eta^2\prec \eta^{3}$: $\tilde{\eta}^{s}i_{5}^s=\eta$.
\end{proof}
\end{lemma}

Lemma \ref{lem:order pi6(X)} can also be proved in terms of the explicit generators of the groups $[X,Y]$ characterized in \cite{lipc22}, where $X,Y$ are elementary $\an2$-complexes. Since \cite{lipc22} still doesn't get published, we adopt the above proof arguments. The following can be easily deduced from  Lemma \ref{lem:order pi6(X)}. 

\begin{lemma}\label{lem:index-orders}
	For positive integers $r'>r,s'>s$ and $r''\geq 1$, there hold 
\[\check{i}^r_{P}\tilde{\eta}_{r}\eta\prec \check{i}^{r'}_{P}\tilde{\eta}_{r'}\eta,\quad 
\bar{i}^r_{P}\tilde{\eta}_{r}\eta\prec\bar{i}^{r'}_{P}\tilde{\eta}_{r'}\eta,\quad 
\hat{i}^{s'}_{n+1}\eta^2\prec\hat{i}^{s}_{n+1}\eta^2,\quad \check{i}^{s',r}_{n+1}\eta^2\prec \check{i}^{s,r''}_{n+1}\eta^2.\]
\end{lemma}

\begin{lemma}\label{lem:reduce plus}
	For $n\geq 3$, there are self-homotopy equivalences $\alpha_r$, $\alpha^s$, $\alpha^s_r$  of $P^{n+1}(2^r)$,  $P^{n+2}(2^s)$, $C^{n+2,s}_r$ respectively such that 
	 \begin{align*}
		 &\alpha_r(i_n^r\eta^3+\tilde{\eta}_r\eta)=\tilde{\eta}_r\eta,\quad \alpha^s(i_{n+1}^s\eta^2+\tilde{\eta}_s)=\tilde{\eta}_s,\\
		 &\alpha^s_r(\check{i}_{n+1}^{s,r}\eta^2+\check{i}^r_P\tilde{\eta}_r\eta)=\check{i}^r_P\tilde{\eta}_r\eta.
	 \end{align*}
	 
 \end{lemma}
 \begin{proof}
	 Take $r'>r$ and let 
	 \begin{align*}
		 \alpha_r=id+i_{n}^r\eta q_{n+1}^r; \quad  \alpha^s=id+i_{n+1}^s\eta q_{n+2}^s;\quad  \alpha^s_r=id+\check{i}_{n+1}^{s,r}q_{n+1}^{r'}\bar{q}_{r'}^r\check{q}_{C}^{r},
	 \end{align*}
	 where  $\check{i}_{n+1}^{s,r}q_{n+1}^{r'}\bar{q}_{r'}^r\check{q}_{C}^{r}$ is the composite 
	 \begin{align*}
		 C_{r}^{n+2,s}\xra{\check{q}_{C}^{r}}C_{r}^{n+2}\xra{\bar{q}_{r'}^{r}}P^{n+1}(2^{r'})\xra{\bar{q}_{n+1}^{r'}}	S^{n+1}	\xra{\check{i}_{n+1}^{s,r}}C_{r}^{n+2,s}
	 \end{align*}
	 with $\bar{q}_{r'}^r$ defined by the diagram (\ref{diagram qrr'}).
	 Notice that $\alpha_r$, $\alpha^s$, $\alpha^s_r$ are self-homotopy equivalences,  since they induce isomorphisms of homology groups of simply-connected spaces.
	 
	 We only check  $\alpha^s_r(\check{i}_{n+1}^{s,r}\eta^2+\check{i}^r_P\tilde{\eta}_r\eta)=\check{i}^r_P\tilde{\eta}_r\eta$ here, since the other two cases are easier.
	 Note that $\check{i}_{n+1}^{s,r}q_{n+1}^{r'}\bar{q}_{r'}^r\check{q}_{C}^{r}\check{i}_{n+1}^{s,r}\eta^2=0$: $\check{q}_{C}^{r}\check{i}_{n+1}^{s,r}=0$ by the homotopy cofibre sequence for $C_{r}^{n+2,s}$ listed in \textbf{Cof. List}. Thus 
	 \begin{align*}
		\alpha^s_r(\check{i}_{n+1}^{s,r}\eta^2+\check{i}^r_P\tilde{\eta}_r\eta)= &(id+\check{i}_{n+1}^{s,r}q_{n+1}^{r'}\bar{q}_{r'}^r\check{q}_{C}^{r})(\check{i}_{n+1}^{s,r}\eta^2+\check{i}^r_P\tilde{\eta}_r\eta)\\
		 =&\check{i}_{n+1}^{s,r}\eta^2+\check{i}^r_P\tilde{\eta}_r\eta+\check{i}_{n+1}^{s,r}q_{n+1}^{r'}\bar{q}_{r'}^r\check{q}_{C}^{r}\check{i}_{n+1}^{s,r}\eta^2\\
		 &+\check{i}_{n+1}^{s,r}q_{n+1}^{r'}\bar{q}_{r'}^r\check{q}_{C}^{r}\check{i}^r_P\tilde{\eta}_r\eta\\
		 =&\check{i}_{n+1}^{s,r}\eta^2+\check{i}^r_P\tilde{\eta}_r\eta+\check{i}_{n+1}^{s,r}q_{n+1}^{r'}\bar{q}_{r'}^r\check{q}_{C}^{r}\check{i}^r_P\tilde{\eta}_r\eta\\
		 =&\check{i}_{n+1}^{s,r}\eta^2+\check{i}^r_P\tilde{\eta}_r\eta+\check{i}_{n+1}^{s,r}q_{n+1}^{r'}\tilde{\eta}_{r'}\eta,\qquad \text{ by }(\ref{equ:qrr'})\\
		 =&\check{i}_{n+1}^{s,r}\eta^2+\check{i}^r_P\tilde{\eta}_r\eta+\check{i}_{n+1}^{s,r}\eta^2\\
		 =&\check{i}^r_P\tilde{\eta}_r\eta.
	 \end{align*}
 \end{proof}

\begin{lemma}\label{order of 3 comp pi8} 
	Let $n\geq 5$ and denote 
	\[U_{(3)}=\{\alpha_{1}(n), i_{n}^{\eta}\alpha_{1}(n) \}\cup \{\hat{i}^{s}_n\alpha_{1}(n)~|~s\geq 1\}.\] 
	 The following relations hold:
	\begin{enumerate}[(1)]
		\item \label{a<a', a'<a}  for any $\alpha, \alpha'\in U_{(3)}$, $\alpha\prec\alpha'$ and $\alpha'\prec\alpha$;
		\item  \label{a<ira} 	for any $\alpha \in U_{(3)}$,  $\alpha\prec i_n^r\alpha_{1}(n)$.
		\item \label{ir'a<ira}  $i_n^{r'}\alpha_{1}(n)\prec i_n^r\alpha_{1}(n)$ ($r<r'$).
	\end{enumerate}
	\begin{proof}
		Clearly  we have  $\alpha_{1}(n)\prec i_{n}^{\eta}\alpha_{1}(n), \hat{i}^{s}_n\alpha_{1}(n)$. So for  (\ref{a<a', a'<a}) of this lemma, we only need to show $i_{n}^{\eta}\alpha_{1}(n), \hat{i}^{s}_n\alpha_{1}(n)\prec \alpha_{1}(n)$ which are obtained by 
		\begin{align}
			&(-\bar{\zeta}_{n})i_{n}^{\eta}\alpha_{1}(n)=-2\alpha_{1}(n)=\alpha_{1}(n); \nonumber\\ &(-\bar{\zeta}_{n}\hat{q}_{\eta})\hat{i}^{s}_n\alpha_{1}(n)=(-\bar{\zeta}_{n})i_{n}^{\eta}\alpha_{1}(n)=\alpha_{1}(n), \label{equ:zeta qeta}
		\end{align}
		where $\bar{\zeta}_{n}$ is given in Lemma \ref{Lem gen of Ceta} and $\hat{q}_{\eta}$ (defined by the \textbf{Cof. List}) satisfies the following homomotopy commutative diagram
		\begin{align*}
			\xymatrix{
				S^{n+1}\ar@{=}[d] \ar[rr]^-{\smatwo{2^s}{\eta}} && S^{n+1}\vee S^n\ar[rr]^{[\hat{i}^s_{n+1},\hat{i}^s_n]} \ar[rr]\ar[d]^-{p_2^{n}}& &C^{n+2,s}\ar[d]^-{\hat{q}_{\eta}} \\
				S^{n+1}\ar[rr]^-{\eta} &&S^{n} \ar[rr]^{i_n^{\eta}} && C_{\eta}^{n+2}
			}.
		\end{align*}
		The statement (\ref{a<ira}) of the Lemma follows by (\ref{a<a', a'<a}) and the obvious relation $\alpha_{1}(n)\prec i_n^r\alpha_{1}(n)$.
		
		For (\ref{ir'a<ira}), the map $B(\chi^{r'}_r)$ for $p=3$ in (\ref{Moore1-PP-odd}) of Lemma \ref{lem:Moore1} satisfies the following homotopy commutative diagram
			\begin{align*}
			\xymatrix{
				S^{n}\ar[d]_-{3^{r'-r}} \ar[r]^-{3^{r'}} &  S^n\ar[r]^-{i_n^{r'}} \ar[r]\ar@{=}[d] &P^{n+1}(3^{r'})\ar[d]^-{B(\chi^{r'}_r)} \\
				S^{n}\ar[r]^-{3^r} &S^{n} \ar[r]^-{i_n^{r}} &P^{n+1}(3^{r}).
			}
		\end{align*}
	 Thus the proof of (\ref{ir'a<ira}) is finished by $B(\chi^{r'}_r)i_n^{r'}\alpha_{1}(n)=i_n^r\alpha_{1}(n)$.
	\end{proof}
\end{lemma}

For maps $\alpha\colon S^{n+3}\to  X$, $\beta\colon S^{n+3}\to  Y$, 
relations $\alpha\prec\beta $ of generators of 2-primary component and 3-primary component of the homotopy groups repectively given in Lemma \ref{lem:order pi6(X)} and Lemma \ref{order of 3 comp pi8} means that we can use self-homotopy  equivalence of $X\vee Y$ to eliminate $\beta$ by $\alpha$.

\begin{lemma}\label{lem:total order 2, 3}
	The generators of  $\pi_{n+3}(S^{n})$,  $\pi_{n+3}(C^{n+2,s})$ and $\pi_{n+3}(C^{n+2,s'})$ ($s'>s$) satisfy the following relations:
\begin{enumerate}
	\item $\alpha_1(n)+\eta^3\prec \alpha_1(n)$, $\alpha_1(n)+\eta^3\prec \eta^3 $, $\alpha_1(n)+\eta^3\prec\hat{i}^{s}_{n}\alpha_1(n)$;
	\item $\hat{i}^{s}_{n}\alpha_1(n)+\hat{i}^{s}_{n+1}\eta^2\prec\alpha_1(n)$, $\hat{i}^{s}_{n}\alpha_1(n)+\hat{i}^{s}_{n+1}\eta^2\prec\eta^3$;
	\item  $\hat{i}^{s}_{n}\alpha_1(n)+\hat{i}^{s}_{n+1}\eta^2\prec \hat{i}^{s}_{n}\alpha_1(n)$,  $\hat{i}^{s}_{n}\alpha_1(n)+\hat{i}^{s}_{n+1}\eta^2\prec \hat{i}^{s}_{n+1}\eta^2$;
	\item $\hat{i}^{s'}_{n}\alpha_1(n)+\hat{i}^{s'}_{n+1}\eta^2\prec \hat{i}^{s}_{n}\alpha_1(n)$, $\hat{i}^{s'}_{n}\alpha_1(n)+\hat{i}^{s'}_{n+1}\eta^2\prec \hat{i}^{s}_{n+1}\eta^2$;
	\item $\hat{i}^{s}_{n}\alpha_1(n)+\hat{i}^{s}_{n+1}\eta^2\prec\hat{i}^{s'}_{n}\alpha_1(n)$.
\end{enumerate}

\begin{proof}
$(1)$ The relations in (1) are deduced from the following equalities indicated by $\hat{i}^{s}_{n}\eta^3=0\in \pi_{n+3}(C^{n+2,s})$:
	\begin{align*}
		&(-2)(\alpha_1(n)+\eta^3)=\alpha_1(n); \quad 3(\alpha_1(n)+\eta^3)=\eta^3;\quad  \hat{i}^{s}_{n}(\alpha_1(n)+\eta^3)=\hat{i}^{s}_{n}\alpha_1(n).
	\end{align*}

 $(2)$ The first relation of (2) is obtained by the equation (see (\ref{equ:zeta qeta}))
		\begin{align*}
		&(-\bar{\zeta}_{n}\hat{q}_{\eta})(\hat{i}^{s}_{n}\alpha_1(n)+\hat{i}^{s}_{n+1}\eta^2)=(-\bar{\zeta}_{n})i_{n}^{\eta}(\alpha_1(n)+0)=\alpha_1(n).
	\end{align*}
	From the exact sequence 
	\begin{align*}
		[C^{n+2,s}, S^n]\xra{[\hat{i}^{s}_{n},\hat{i}^s_{n+1}]^\ast} [S^{n}\vee S^{n+1}, S^n]\xra{\smatwo{\eta}{2^s}^{\ast}}  [S^{n+1}, S^n],
	\end{align*}
we see that there is a map $\mu^s\in [C^{n+2,s}, S^n]$ satisfying \[\mu^s[\hat{i}^{s}_{n},\hat{i}^{s}_{n+1}]=[0,\eta].\]
  Then second relation of (2) follows by the equalities 
	\begin{align*}
		&\mu^s(\hat{i}^{s}_{n}\alpha_1(n)+\hat{i}^{s}_{n+1}\eta^2)=\mu^s[\hat{i}^{s}_{n},\hat{i}^s_{n+1}]\smatwo{\alpha_1(n)}{\eta}=[0,\eta] \smatwo{\alpha_1(n)}{\eta}=\eta^3.
	\end{align*}

	$(3)$ The relations in $(3)$ can be obtained by the two equalities 
	\begin{align*}
	&(-2)(\hat{i}^{s}_{n}\alpha_1(n)+\hat{i}^{s}_{n+1}\eta^2)=\hat{i}^{s}_{n}\alpha_1(n); \quad 3(\hat{i}^{s}_{n}\alpha_1(n)+\hat{i}^{s}_{n+1}\eta^2)=\hat{i}^{s}_{n+1}\eta^2.
\end{align*}

$(4)$ There are maps $\mu_{s}^{s'}, \lambda_{s}^{s'}\in [C^{n+2,s'}, C^{n+2,s}]$ respectively satisfying  the following homotopy commutative diagrams 
	\[
	\xymatrix{
		S^{n+1}\ar[d]^-{-2^{s'+1-s}} \ar[rr]^-{\smatwo{2^{s'}}{\eta}} && S^{n+1}\vee S^n\ar[rr]^{[\hat{i}^{s'}_{n+1},\hat{i}^{s'}_n]} \ar[d]^-{-2id} &&C^{n+2,s'}\ar[d]^-{\mu_{s}^{s'}} \\
		S^{n+1}\ar[rr]^-{\smatwo{2^{s}}{\eta}}  &&S^{n+1}\vee S^n\ar[rr]^{[\hat{i}^{s}_{n+1},\hat{i}^{s}_n]} && C^{n+2,s}
	};\]
\[\xymatrix{
	S^{n+1}\ar[d]^-{-2^{s'-s}} \ar[rr]^-{\smatwo{2^{s'}}{\eta}} && S^{n+1}\vee S^n\ar[rr]^{[\hat{i}^{s'}_{n+1},\hat{i}^{s'}_n]} \ar[d]^-{id\vee 0} &&C^{n+2,s'}\ar[d]^-{\lambda_{s}^{s'}} \\
	S^{n+1}\ar[rr]^-{\smatwo{2^{s}}{\eta}}  &&S^{n+1}\vee S^n\ar[rr]^{[\hat{i}^{s}_{n+1},\hat{i}^{s}_n]} && C^{n+2,s}
	}.\]
Then the proof of (4) is finished by the following equalities  
\begin{align*}
	\mu_{s}^{s'}(\hat{i}^{s'}_{n}\alpha_1(n)+\hat{i}^{s'}_{n+1}\eta^2)=\hat{i}^{s}_{n}\alpha_1(n); \quad  	\lambda_{s}^{s'}(\hat{i}^{s'}_{n}\alpha_1(n)+\hat{i}^{s'}_{n+1}\eta^2)=\hat{i}^{s}_{n+1}\eta^2.
\end{align*} 

$(5)$ There is a map $\theta_{s'}^{s}\in  [C^{n+2,s}, C^{n+2,s'}]$ satisfying the following homotopy commutative diagram
\begin{align*}
	\xymatrix{
		S^{n+1}\ar@{=}[d] \ar[rr]^-{\smatwo{2^{s}}{\eta}} && S^{n+1}\vee S^n\ar[rr]^{[\hat{i}^{s}_{n+1},\hat{i}^{s}_n]} \ar[d]^-{2^{s'-s}\vee id} &&C^{n+2,s}\ar[d]^-{\theta_{s'}^{s}} \\
		S^{n+1}\ar[rr]^-{\smatwo{2^{s'}}{\eta}}  &&S^{n+1}\vee S^n\ar[rr]^{[\hat{i}^{s'}_{n+1},\hat{i}^{s'}_n]} && C^{n+2,s'}
	}.
\end{align*}
Then the proof of (5) is finished by the equalities
\begin{align*}
	\theta_{s'}^{s}(\hat{i}^{s}_{n+1}\eta^2+\hat{i}^{s}_{n}\alpha_1(n))&=[\hat{i}^{s'}_{n+1},\hat{i}^{s'}_n](2^{s'-s}\vee id)\smatwo{\eta^2}{\alpha_1(n)} \\
	&=[\hat{i}^{s'}_{n+1}2^{s'-s},\hat{i}^{s'}_n]\smatwo{\eta^2}{\alpha_1(n)}\\
	&=\hat{i}^{s'}_{n}\alpha_1(n).
\end{align*}
\end{proof}
\end{lemma}

\subsection{Proofs of Theorem \ref{thm:6mfds 2-local} and \ref{thm:6mfds total}}\label{sect:proofs}
We now give the proofs of the main theorems. 

\begin{proof}[Proof of Theorem \ref{thm:6mfds 2-local}]
Denote 
\[V''_7=\Sigma V''_6/(\bigvee_{j=1}^{m_3}P^{6}(3^{r'_j})\vee kC_{\eta}^7).\]
Localized at $2$, by (\ref{eq:S2M'}), (\ref{equ:Sigma g7}) and \cite[Lemma 6.4]{HL}, there is a homotopy equivalence 
\begin{equation}\label{eq:S3M':2-local}
	\Sigma^3 M'\simeq_{(2)} kC_{\eta}^7\vee \bm C_{g_{(2)}}
\end{equation}
for some map $g_{(2)}\colon S^{8}\to V''_7$ given by the equation
\begin{equation}\label{equ:g(2)}
	\begin{aligned}
g_{(2)}	=&\sum_{j=1}^{l-k-t_3} c^S_j\cdot \eta^3+\sum_{j=1}^{l-k-t_2} d^S_j\cdot \eta+\sum_{j=1}^{t_0}\big(x_j\cdot \tilde{\eta}_{s_j}+y_j\cdot i_{6}^{s_j}\eta^2\big) \\
&+\sum_{j=1}^{t_1}\big(a_j\cdot \tilde{\eta}_{r_j}\eta+c_j\cdot i_{5}^{r_j}\eta^{3}\big)+\sum_{j=1}^{t_2}\big(\bar{a}_j\cdot \bar{i}_{P}^{\bar{r}_j}\tilde{\eta}_{\bar{r}_j}\eta\big) \\
&+\sum_{j=1}^{t_3}\big(\hat{y}_j\cdot \hat{i}^{\hat{s}_j}_6\eta^2\big)+\sum_{j=1}^{t_4}\big(\check{a}_j\cdot \check{i}_{P}^{\check{r}_j}\tilde{\eta}_{\check{r}_j}\eta+\check{y}_j\cdot\check{i}^{\check{s}_j,\check{r}_j}_6\eta^2\big). 
	\end{aligned}
\end{equation}
We sort the coefficients in (\ref{equ:g(2)}) by setting
\begin{align*}
	N_1&=\{d^S_j~ |~ 1 \leq j\leq 1-k-t_2\}\cup \{x_j~ |~ 1 \leq j\leq t_0\},\\
	N_2&=\{y_j~ |~ 1 \leq j\leq t_0\}\cup \{a_j~ |~ 1 \leq j\leq t_1\}\cup \{\bar{a}_j~ |~ 1 \leq j\leq t_2\}\cup N_2^2,\\
	&N_2^2= \{\hat{y}_j~ |~ 1 \leq j\leq t_3\}\cup  \{\check{a}_j, \check{y}_j~ |~ 1 \leq j\leq t_4\},\\
	N_3&=\{	c^S_j~ |~ 1 \leq j\leq 1-k-t_3\}\cup  \{c_j~ |~ 1 \leq j\leq t_1\}. 
\end{align*}
Note that the corresponding generators related to elements of $N_1,N_2$ and $N_3$ are detected by the reduced Steenrod square $\Sq^2$, the second operation $\Theta$ and the triple operation $\mathbb{T}$, respectively.
We discuss the homotopy decompositions of the homotopy cofibre $\bm C_{g_{(2)}}$ as follows.

(1) Suppose that the operations $\Sq^2$,  $\Theta$ and  $\mathbb{T}$ act trivially on  $H^4(M;\z{})$, $H^3(M;\z{})$  and  $H^2(M;\z{})$ respectively, then all coefficients in (\ref{equ:g(2)}) are zero, which implies that 
$ \bm C_{g_{(2)}}\simeq S^9\vee V''_7$.

(2)  Suppose that $\Sq^2$ acts trivially on $H^4(M;\z{})$,  then all coefficients in $N_1$ are zero. Moreover, $\Theta$ acts non-trivially on $H^3(M;\z{})$, we get at least one of the coefficients in  $N_2$ equals $1$. By  Lemma \ref{lem:order pi6(X)} and Lemma \ref{lem:reduce plus}, after applying some self-homotopy equivalence of $V''_7$,  we can assume that  there is exactly one such coefficient in $N_2$; i.e.,  $y_{j_0}=1$ or $a_{j_0}=1$ or $\bar{a}_{j_0}=1$ or $\hat{y}_{j_0}=1$ or $\check{a}_{j_0}=1$ or $\check{y}_{j_0}=1$ for some $j_0$. In any cases we have that all coefficients in $N_3$ are zero. Thus
 the homotopy type of $\bm C_{g_{(2)}}$ is one of the following 
\begin{align*}
\big(V''_7/P^{7}(2^{s_{j_0}})\big)\vee \big(P^{7}(2^{s_{j_0}})\cup_{i_6^{s_{j_0}}\eta^2}e^9\big),\quad &\big(V''_7/P^{6}(2^{r_{j_0}})\big)\vee \big(P^{6}(2^{r_{j_0}})\cup_{\tilde{\eta}_{r_{j_0}}\eta}e^9\big),\\[1ex]
\big(V''_7/C^{7}_{\bar{r}_{j_0}}\big)\vee \big(C^{7}_{\bar{r}_{j_0}}\cup_{\bar{i}_P^{\bar{r}_{j_0}}\tilde{\eta}_{\bar{r}_{j_0}}\eta}e^9\big),\quad &\big(V''_7/C^{7,\hat{s}_{j_0}}\big)\vee \big(C^{7,\hat{s}_{j_0}}\cup_{\hat{i}^{\hat{s}_{j_0}}_6\eta^2}e^9\big),\\[1ex]
\big(V''_7/C^{7,\check{s}_{j_0}}_{\check{r}_{j_0}}\big)\vee \big(C^{7,\check{s}_{j_0}}_{\check{r}_{j_0}}\cup_{\check{i}_P^{\check{r}_{j_0}}\tilde{\eta}_{\check{r}_{j_0}}\eta}e^9\big),\quad &\big(V''_7/C^{7,\check{s}_{j_0}}_{\check{r}_{j_0}}\big)\vee \big(C^{7,\check{s}_{j_0}}_{\check{r}_{j_0}}\cup_{\check{i}^{\check{s}_{j_0},\check{r}_{j_0}}_6\eta^2}e^9\big).
\end{align*}
	
(3) Suppose that $\Sq^2$ and  $\Theta$  act trivially on $H^4(M;\z{})$ and $H^3(M;\z{})$, respectively, then all coefficients in $N_1\cup N_2$ are zero.  Since $\mathbb{T}$ acts non-trivially on $H^2(M;\z{})$,  at least one of the coefficients in $N_3$ equals $1$. By  Lemma \ref{lem:order pi6(X)}, after applying some self-homotopy equivalence of $V''_7$,  we can assume that there is exactly one such coefficient in $N_3$; i.e.,  $c^S_{j_0}=1$ or $c_{j_0}=1$ for some $j_0$ and all other coefficients in $N_3$ are zero. Thus we have 
	 \begin{align*}
		&\bm C_{g_{(2)}}\simeq \big(V''_7/S^5\big)\vee \big(S^5\cup_{\eta^3}e^9\big),\\
\text{ or }\quad &\bm C_{g_{(2)}}\simeq \big(V''_7/P^{6}(2^{r_{j_0}})\big)\vee \big(P^{6}(2^{r_{j_0}})\cup_{i_5^{r_{j_0}}\eta^3}e^9\big)~ (r_{j_0}\geq 3).
	 \end{align*}	

(4)	 Suppose that $\Sq^2$ acts non-trivially on $H^4(M;\z{})$,   then at least one of the coefficients in  $N_1$ must be $1$.
	By  Lemma \ref{lem:order pi6(X)} and Lemma \ref{lem:reduce plus}, after applying some self-homotopy equivalence of $V''_7$,  we can assume that there is exactly one such coefficient in $N_1$; i.e.,  $x_{j_0}=1$ or $d^S_{j_0}=1$ for some $j_0$ and all other coefficients in (\ref{equ:g(2)}) are zero. It follows that  
	\[\bm C_{g_{(2)}}\simeq \big(V''_7/P^{7}(2^{s_{j_0}})\big)\vee \big(P^{7}(2^{s_{j_0}})\cup_{\tilde{\eta}_{s_{j_0}}}e^9\big)~~\text{ or }~~\big(V''_7/S^7\big)\vee C_{\eta}^9.\]

Now the proof of Theorem \ref{thm:6mfds 2-local} is finished by combining (\ref{equ:M=M'vee dS4}) and (\ref{eq:S3M':2-local}).
\end{proof}


Lemma \ref{lem:total order 2, 3}  shows that we can apply appropriate self-homotopy equivalences to eliminate some generators of the $2$-primary components of the groups given by Lemma \ref{lem:order pi6(X)} without affecting the $3$-primary component generators in Lemma \ref{order of 3 comp pi8}, and vice versa. Thus 
we can combine the $2$-local and $3$-local homotopy types of $\Sigma^3 M$ to get its total homotopy types. 


\begin{theorem}[cf. Theorem 1.1 of \cite{CS22}]\label{thm:3-local}
	Let $M$ be a closed oriented simply-connected $6$-manifold with $H_\ast(M)$ given by (\ref{HM}). 
	\begin{enumerate}[1.]
	   \item\label{3-local:1} If $\PP\colon H^{3}(\Sigma M;\Z/3)\to H^{7}(\Sigma M;\Z/3)$ is trivial, then there is a homotopy equivalence
   \[\Sigma^3 M\simeq_{(3)}dS^6\vee P^{7}(T_{3})\vee  P^{6}(T_{3})\vee l(S^5\vee S^7)\vee S^9.\]
   \item\label{3-local:2} If for any cohomology classes $u,v\in H^3(\Sigma M;\Z/3)$ with $\PP(u)\neq 0$, there holds $\beta_r(u+v)=0$ for any $r\geq 1$, then there is a homotopy equivalence 
	\[\Sigma^3 M\simeq_{(3)}dS^6\vee P^{7}(T_{3})\vee  P^{6}(T_{3})\vee (l-1)S^5\vee lS^7\vee (S^5\cup_{\alpha_1(5)}e^9).\]
   
	\item\label{3-local:3} If there exist cohomology classes $u,v\in H^3(\Sigma M;\Z/3)$ such that $\PP(u)\neq 0$ and $\beta_{r_{j_0}}(u+v)\neq 0$ for some $r_{j_0}$, then there is a homotopy equivalence 
	\[ \Sigma^3 M\simeq_{(3)}  dS^6\vee P^{7}(T_{3})\vee  P^{6}\big(\frac{T_{3}}{\Z/3^{r_{j_0}}}\big)\vee l(S^5\vee S^7)\vee \big(P^{6}(3^{r_{j_0}})\cup_{i_{5}^{r_{j_0}}\alpha_1(5)}e^9\big).\]
	\end{enumerate}
   \end{theorem}
 Note that one can desuspend the homotopy equivalences in Theorem \ref{thm:3-local} twice to get the homotopy type of $\Sigma M$ \cite{CS22}.

\begin{proof}[Proof of Theorem \ref{thm:6mfds total}] 
In the expression (\ref{equ:Sigma g7}), we may put
\begin{enumerate}[(1)]
	\item $4c^S_j\cdot \nu=c^2_j\cdot \eta^3+ c^3_j\cdot \alpha_1(5)$ for some $c^2_j\in \{0,1\},c^3_j\in \{0,1,2\}$;
	\item $4c^{\eta}_ji_5^{\eta} \nu=c^{\eta}_ji_5^{\eta}\alpha_1(5)$, $4\hat{c}_j\hat{i}^{\hat{s}_{j}}_5 \nu=\hat{c}_j\hat{i}^{\hat{s}_{j}}_5 \alpha_1(5)$  (or $4c^{\eta}_ji_5^{\eta} \nu=-c^{\eta}_ji_5^{\eta}\alpha_1(5)$, $4\hat{c}_j\hat{i}^{\hat{s}_{j}}_5 \nu=-\hat{c}_j\hat{i}^{\hat{s}_{j}}_5 \alpha_1(5)$, then we also denote the coefficients   $-c^{\eta}_j$ and $-\hat{c}_j$ by $c^{\eta}_j$ and $\hat{c}_j$ respectively).
\end{enumerate}
 Then (\ref{equ:Sigma g7}) becomes 
 \begin{equation} \label{equ:Sigma g7 2,3 part} 
		\begin{aligned}
		\Sigma h''	=&\sum_{j=1}^{m_3} a'_j\cdot i_5^{r'_j}\alpha_{1}(5)+\sum_{j=1}^{l-k-t_3} \big(c^2_j\cdot\eta^3+ c^3_j\alpha_1(5)\big)+\sum_{j=1}^{k}c^{\eta}_j\cdot i_5^{\eta} \alpha_1(5)\\
		&+\sum_{j=1}^{l-k-t_2} d^S_j\cdot \eta+\sum_{j=1}^{t_0}\big(x_j\cdot \tilde{\eta}_{s_j}+y_j\cdot i_{6}^{s_j}\eta^2\big)+\sum_{j=1}^{t_1}\big(a_j\cdot \tilde{\eta}_{r_j}\eta+c_j\cdot i_{5}^{r_j}\eta^{3}\big)\\
		&+\sum_{j=1}^{t_2}\big(\bar{a}_j\cdot \bar{i}_{P}^{\bar{r}_j}\tilde{\eta}_{\bar{r}_j}\eta\big)+\sum_{j=1}^{t_3}\big(\hat{c}_j\cdot \hat{i}^{\hat{s}_{j}}_5 \alpha_1(5)+\hat{y}_j\cdot \hat{i}^{\hat{s}_j}_6\eta^2\big)\\
		&+\sum_{j=1}^{t_4}\big(\check{a}_j\cdot \check{i}_{P}^{\check{r}_j}\tilde{\eta}_{\check{r}_j}\eta+\check{y}_j\cdot\check{i}^{\check{s}_j,\check{r}_j}_6\eta^2\big). 
	\end{aligned}
 \end{equation}
Here we only give the proofs of the cases (\ref{Thm:P1 neq 0:b:4}),  (\ref{Thm:P1 neq 0:b:5}), (\ref{Thm:P1 neq 0:c:1}) and (\ref{Thm:P1 neq 0:c:2}) of Theorem \ref{thm:6mfds total},  the proofs of other cases are similar or easier. 
\medskip 

\underline{Proof of the cases  (\ref{Thm:P1 neq 0:b:4}) and  (\ref{Thm:P1 neq 0:b:5}).}	
Suppose that $\Sq^2$ acts trivially on $H^4(M;\z{})$ and that $\Theta$ acts non-trivially on $H^3(M;\z{})$. From the proof of case (3) of Theorem \ref{thm:6mfds 2-local} (assuming $\hat{y}_{j_0}=1$), we know that there is a self-homotopy equivalence of $\Sigma V''_6$ such that 
\begin{equation}\label{equ:g8} 
		\begin{aligned}
		 	\Sigma h''\sim \hbar=&\sum_{j=1}^{m_3} a'_j\cdot i_5^{r'_j}\alpha_{1}(5)+\sum_{j=1}^{l-k-t_3} c^3_j\cdot \alpha_1(5)+\sum_{j=1}^{k} c^{\eta}_j\cdot i_5^{\eta} \alpha_1(5)\\
		&+\sum_{j=1,j\neq j_0}^{t_3}\big(\hat{c}_j\cdot \hat{i}^{\hat{s}_{j}}_5 \alpha_1(5)+\hat{c}_{j_0}\cdot \hat{i}^{\hat{s}_{j_0}}_5 \alpha_1(5)+\hat{i}^{\hat{s}_{j_0}}_6\eta^2\big).  
	\end{aligned}
\end{equation}	
	Then (\ref{equ:g8}) can be obtained since we can choose certain a self-homotopy equivalence of $\Sigma V''_6$ which does not affect the $3$-torsion summation terms of (\ref{equ:Sigma g7 2,3 part}), by Lemma \ref{lem:total order 2, 3}.
	Since $\PP\colon H^{3}(\Sigma M;\Z/3)\to H^{7}(\Sigma M;\Z/3)$ is non-trivial, at least one of $a'_j$ , $c_j^3$, $c^{\eta}_j$, $\hat{c}_j$ is nonzero. 
	
	If  \textbf{Condition $\star$} holds, then by the proof of \cite[Proposition 5.2]{lipc23},  we have  
	\[c_j^3= c^{\eta}_j=\hat{c}_j \text{ for all $j$ and }   a'_{j'_0}=1\text{ for some $j'_0$}.\]  
	 Moreover, by Lemma \ref{order of 3 comp pi8} $(iii)$,  we can assume $a'_{j}=0$ for all $j\neq j'_0$. It follows  that 
    \begin{align}
\Sigma h''\sim \hbar\sim \hat{i}^{\hat{s}_{j_0}}_6\eta^2+ i_5^{r'_{j'_0}}\alpha_{1}(5), \label{sim:a'j0=1}
    \end{align}
which implies a homotopy equivalence
	 \[\bm C_{\Sigma h''}\simeq\big(\Sigma V''_6/(C^{7,\hat{s}_{j_0}}\vee P^{6}(3^{r'_{j'_0}}))\big)\vee \big((C^{7,\hat{s}_{j_0}}\vee  P^{6}(3^{r'_{j'_0}}))\cup_{\smatwo{\hat{i}^{\hat{s}_{j_0}}_6\eta^2}{i_5^{r'_{j'_0}}\alpha_{1}(5)}}e^9\big).\]
	
		If  \textbf{Condition $\star$} does not hold, then we may assume that one of $c_j^3$, $c^{\eta}_j$, $\hat{c}_j$ is nonzero; by Lemma \ref{order of 3 comp pi8} (ii), we assume  $a'_j=0$ for any $j$.
		
	 Case 1.  If $l-k-t_3\neq 0$, then there is a wedge summand $S^5$ in $\Sigma V''_6$. By (i) of  Lemma \ref{order of 3 comp pi8}, we get 
	 \begin{align}
	 	\matwo{0}{\alpha_X}\sim \matwo{\alpha_{1}(5)}{0}\colon S^8\to  S^5\vee X  \label{sim:S5vee X}
	 \end{align}
	where $\alpha_X=\alpha_{1}(5), i_5^{\eta}\alpha_{1}(5), \hat{i}^{\hat{s}_{j}}_5\alpha_{1}(5)$ for $X=S^5, C_{\eta}^7, C^{7,\hat{s}_{j}}$. So 
		\[\Sigma h''\sim \hbar\sim \hat{i}^{\hat{s}_{j_0}}_6\eta^2+\alpha_{1}(5),\]
		which implies  
		\[\bm C_{\Sigma h''}\simeq \big(\Sigma V''_6/(C^{7,\hat{s}_{j_0}}\vee S^5)\big)\vee \big((C^{7,\hat{s}_{j_0}}\vee  S^5)\cup_{\smatwo{\hat{i}^{\hat{s}_{j_0}}_6\eta^2}{\alpha_{1}(5)}}e^9\big).\]
		
	 Case 2. If $l-k-t_3=0$, $k \neq 0$, i.e,  there is no wedge summand $S^5$, but 
	  a wedge summand $C_{\eta}^7$ in $\Sigma V''_6$, then 
	   by Lemma \ref{order of 3 comp pi8} (i), we get 
	  \[ \matwo{0}{\alpha_X}\sim \matwo{i_{5}^{\eta}\alpha_{1}(5)}{0}\colon S^8\to   C_{\eta}^7\vee X\]
	  where $\alpha_X=i_5^{\eta}\alpha_{1}(5), \hat{i}^{\hat{s}_{j}}_5\alpha_{1}(5)$ for $X=C_{\eta}^7, C^{7,\hat{s}_{j}}$. So 
	  \[\Sigma h''\sim \hbar\sim \hat{i}^{\hat{s}_{j_0}}_6\eta^2+i_5^{\eta}\alpha_{1}(5),\]
	  which implies 
	  \[\bm C_{\Sigma h''}\simeq \big(\Sigma V''_6/(C^{7,\hat{s}_{j_0}}\vee C_{\eta}^7)\big)\vee \big((C^{7,\hat{s}_{j_0}}\vee  C_{\eta}^7)\cup_{\smatwo{\hat{i}^{\hat{s}_{j_0}}_6\eta^2}{i_5^{\eta}\alpha_{1}(5)}}e^9\big).\]
	  
	 Case 3.  If $l-k-t_3=0$ and  $k=0$, i.e,  there are no wedge summands $S^5$ and $C_{\eta}^7$ in $\Sigma V''_6$, then at least one 
	  $\hat{c}_{j}\neq 0$.   By (i) of  Lemma \ref{order of 3 comp pi8} and  Lemma \ref{lem:total order 2, 3}, we get 
	   \[ \matwo{\hat{i}^{\hat{s}_{j_0}}_6\eta^2+\hat{c}_{j_0}\cdot\hat{i}^{\hat{s}_{j_0}}_5\alpha_{1}(5)}{\hat{i}^{\hat{s}_{j}}_5\alpha_{1}(5)}\sim \matwo{\hat{i}^{\hat{s}_{j_0}}_6\eta^2+\hat{i}^{\hat{s}_{j_0}}_5\alpha_{1}(5)}{0}\colon S^8\to   C^{7,\hat{s}_{j_0}}\vee C^{7,\hat{s}_{j}},\]
	    which implies a homotopy equivalence 
	   \[\bm C_{\Sigma h''}\simeq \big(\Sigma V''_6/C^{7,\hat{s}_{j_0}}\big)\vee \big(C^{7,\hat{s}_{j_0}}\cup_{\hat{i}^{\hat{s}_{j_0}}_6\eta^2+\hat{i}^{\hat{s}_{j_0}}_5\alpha_{1}(5)}e^9\big).\]
	Thus we complete the proof of the cases  (\ref{Thm:P1 neq 0:b:4}),  (\ref{Thm:P1 neq 0:b:5}) of Theorem \ref{thm:6mfds total} by the homotopy equivalence
	\begin{align}
W_7\simeq\Sigma V''_6\vee dS^6\vee P^6(T_{\geq 5})\vee P^7(T_{\geq 3}). \label{equ:W7 V6''}
	\end{align}

\underline{Proof of the cases (\ref{Thm:P1 neq 0:c:1}) and  (\ref{Thm:P1 neq 0:c:2}).}	
	Suppose that $\Sq^2$ acts trivially on $H^4(M;\z{})$, $\Theta$ acts trivially on $H^4(M;\z{})$,  and that $\mathbb{T}$ acts non-trivially on $H^2(M;\z{})$. Then by the proof of the case (\ref{thm-spin Theta=0, Tneq0}) of Theorem \ref{thm:6mfds 2-local} (assume  $c^S_{j_0}=1$),  there is a self-homotopy equivalence of $\Sigma V''_6$ such that 
	\begin{align}
		 	\Sigma h''\sim \hslash=&\sum_{j=1}^{m_3} a'_j\cdot i_5^{r'_j}\alpha_{1}(5)+\sum_{j=1,j\neq j_0}^{l-k-t_3} c^3_j\alpha_1(5)+\eta^3 +c^3_{j_0}\alpha_1(5) \nonumber\\
		&+\sum_{j=1}^{k} c^{\eta}_j\cdot i_5^{\eta} \alpha_1(5)+\sum_{j=1}^{t_3}\hat{c}_j\cdot \hat{i}^{\hat{s}_{j}}_5 \alpha_1(5). \label{equ:g'8}  
	\end{align}

	  If \textbf{Condition $\star$} does not hold,  the assumption that $\PP$ acts non-trivially on $H^{3}(\Sigma M;\Z/3)$ implies that at least one of $c_j^3$, $c^{\eta}_j$, $\hat{c}_j$ is nonzero and   $a'_j=0$ for any $j$.  
	Since by the assumption that $c^S_{j_0}=1$, $S^5$ is a wedge summand of $\Sigma V''_6$, by (\ref{sim:S5vee X}), we get 
	\[\Sigma h''\sim \hslash\sim \eta^3 +\alpha_1(5).\]
	Hence we have 
	\[\bm C_{\Sigma h''}\simeq\big(\Sigma V''_6/S^5\big)\vee \big(S^5\cup_{\eta^3 +\alpha_1(5)}e^9\big)\simeq\big(\Sigma V''_6/S^5\big)\vee \big(S^5\cup_{4\nu}e^9\big).\] 
	
	If \textbf{Condition $\star$} holds, then all of $c_j^3$, $c^{\eta}_j$, $\hat{c}_j$ are zero for any $j$ and   $a'_{j'_0}=1$ for some  $j'_0$. Then similar arguments to that of (\ref{sim:a'j0=1}) show that 
	\[\Sigma h''\sim \hslash\sim \eta^3+i_{5}^{r'_{j'_0}}\alpha_1(5),\]
	 which implies a homotopy equivalence
	\[\bm C_{\Sigma h''}\simeq\big(\Sigma V''_6/(S^5\vee P^{6}(3^{r'_{j'_0}}))\big)\vee\big((S^5\vee P^{6}(3^{r'_{j'_0}}))\cup_{\smatwo{\eta^3}{i_{5}^{r'_{j'_0}}\alpha_{1}(5)}}\!\!e^9\big).\]
	We complete the proofs of the cases  (\ref{Thm:P1 neq 0:b:4}),  (\ref{Thm:P1 neq 0:b:5}) of Theorem \ref{thm:6mfds total} by applying (\ref{equ:W7 V6''}).
\end{proof}

\bibliographystyle{amsplain}
\bibliography{refs6mfd}

\end{document}